\documentclass{amsart}
\RequirePackage[colorlinks,citecolor=blue,urlcolor=blue,linkcolor=blue]{hyperref}

 \usepackage{tikz}
 \usetikzlibrary{patterns}
\usepackage{amsmath,amsthm,amsfonts,amssymb}
\usepackage{soul}
 \usepackage{url,enumerate,soul}
\newcommand{\comment}[1]{}
\newcommand{\commYZ}[1]{}
\newcommand{\commJW}[1]{}
\newcommand{\longcomment}[1]{}
\newcommand{\hide}[1]{\fbox{\tiny \color{red} Hidden text}}
\renewcommand{\hide}[1]{\hrule  \fbox{\tiny \color{red} xtra begin} {\color{gray}\tiny #1 } \fbox{\tiny \color{red} xtra end}  \hrule}
\renewcommand{\longcomment}[1]{\fbox{\begin{minipage}{.95\textwidth}\color{magenta}#1\end{minipage}}}
\renewcommand{\comment}[1]{{\fbox{\footnotesize\color{magenta}#1}}}
\renewcommand{\commJW}[1]{\ovalbox{\footnotesize  \color{magenta} \hl{JW}: #1} }
\renewcommand{\commYZ}[1]{\ovalbox{\footnotesize  \color{magenta} \hl{YZ}: #1} }

\def\fddto{\xrightarrow{\textit{f.d.d.}}}
\newcommand{\ind}{{\bf 1}}

\def\inddd#1{{\ind}_{\left\{#1\right\}}}

\newcommand{\proba}{\mathbb P}
\newcommand{\esp}{{\mathbb E}}

\newcommand{\eqnh}{\begin{eqnarray*}}
\newcommand{\eqne}{\end{eqnarray*}}
\newcommand{\eqnhn}{\begin{eqnarray}}
\newcommand{\eqnen}{\end{eqnarray}}
\newcommand{\equh}{\begin{equation}}
\newcommand{\eque}{\end{equation}}

\def\summ#1#2#3{\sum_{#1 = #2}^{#3}}
\def\prodd#1#2#3{\prod_{#1 = #2}^{#3}}

\newcommand{\eqd}{\stackrel{d}{=}}

\def\topp#1{^{(#1)}}

\def\abs#1{\left|#1\right|}

\def\ccbb#1{\left\{#1\right\}}
\def\sccbb#1{\{#1\}}
\def\pp#1{\left(#1\right)}
\def\spp#1{(#1)}
\def\bb#1{\left[#1\right]}

\def\mmid{\;\middle\vert\;}

\def\floor#1{\left\lfloor #1 \right\rfloor}
\def\sfloor#1{\lfloor #1 \rfloor}
\def\ceil#1{\left\lceil #1 \right\rceil}

\def\aa#1{\left\langle #1\right\rangle}

\def\vv#1{{\boldsymbol #1}}

\def\mand{\mbox{ and }}
\def\qmand{\quad\mbox{ and }\quad}

\def\qmwith{\quad\mbox{ with }\quad}
\def\mfa{\mbox{ for all }}

\def\mmas{\mbox{ as }}

\def\wt#1{\widetilde{#1}}

\def\what#1{\widehat{#1}}

\def\limn{\lim_{n\to\infty}}

\def\weakto{\Rightarrow}

\def\R{{\mathbb R}}

\def\N{{\mathbb N}}

\def\BB{{\mathbb B}}

\def\qp#1{(#1;q)_\infty}

\def\qps#1{(#1;q)}

\newcommand{\hidecomment}[1]{}
\theoremstyle{plain}
\newtheorem{theorem}{Theorem}[section]

\newtheorem{lemma}[theorem]{Lemma}
\newtheorem{proposition}[theorem]{Proposition}

\newtheorem{assumption}[theorem]{Assumption}

\theoremstyle{remark}
\newtheorem{definition}[theorem]{Definition}
\newtheorem{remark}[theorem]{Remark}

\def\<{\langle}
\def\>{\rangle}

\usepackage{dsfont}
\newcommand{\RR}{\mathbb{R}}

\newcommand{\NN}{\mathbb{N}}

\newcommand{\eps}{\varepsilon}

\def\p{\mathsf p}
\def\g{{\color{blue}\mathsf q}}
\def\A{{\mathsf a}}

\def\C{{\mathsf c}}

\def\A{{\color{blue}\mathsf a}}

\def\C{{\color{blue} \mathsf c}}

\def\vvx{{\vv x}}
\def\vvc{{\vv c}}
\def\vvu{{\vv u}}

\usepackage{graphicx,fancybox,pgf,amsmath,float,amssymb}
\usepackage{fancybox,graphicx,dsfont,pgf}
\usepackage{subfigure,color,wrapfig}

\usepackage{latexsym}

\renewcommand{\proba}{\mathds{P}}

\newcommand{\erfc}{{\rm erfc}}

\def\<{\langle}
\def\>{\rangle}

\numberwithin{equation}{section}

\def\eps{\varepsilon}
 \newcommand{\arxivc}[1]{ \hrule
 \tiny #1  \normalsize \hrule
 }

\renewcommand{\comment}[1]{}
\renewcommand{\longcomment}[1]{}
\renewcommand{\arxivc}[1]{}

\begin{document}\sloppy
%

\title[ASEP and KPZ fixed point]{From the asymmetric simple exclusion processes \\ to the stationary measures of the KPZ fixed point on an interval}

\author{W\l odek Bryc}
\address
{
W\l odzimierz Bryc\\
Department of Mathematical Sciences\\
University of Cincinnati\\
2815 Commons Way\\
Cincinnati, OH, 45221-0025, USA.
}
\email{wlodek.bryc@gmail.com}

\author{Yizao Wang}
\address
{
Yizao Wang\\
Department of Mathematical Sciences\\
University of Cincinnati\\
2815 Commons Way\\
Cincinnati, OH, 45221-0025, USA.
}
\email{yizao.wang@uc.edu}
\author{Jacek Weso\l owski}
\address{%
Jacek Weso\l owski\\
 Faculty of Mathematics and Information Science
Warsaw University of Technology and Statistics Poland,
Warszawa, Poland}
\email{wesolo@mini.pw.edu.pl}

\begin{abstract}
Barraquand and Le~Doussal \cite{barraquand2022steady} introduced a family of stationary measures  for the (conjectural) KPZ fixed point on an interval with Neumann boundary conditions,
and predicted that they  arise  as  scaling limit of stationary measures of all models in the KPZ %
universality
class on an interval.
In this paper, we show that the stationary measures for KPZ fixed point on an interval %
arise as %
the scaling limits of the %
height increment processes for
the open asymmetric simple exclusion process %
in the steady state, with parameters changing appropriately as the size of the system tends to infinity. %
\end{abstract}

%
%


\maketitle
\arxivc{This is an expanded version of the paper with additional details and calculations.}

\section{Introduction and main result} %
\subsection{KPZ fixed point on an interval} \label{Sec:KPZ}

The asymmetric simple exclusion process (ASEP)  in one dimension is one of the most widely investigated models for open non-equilibrium systems in the physics literature  and serves as a basic model in the Kardar--Parisi--Zhang (KPZ)    universality class.
In particular,
investigations on the connection between the ASEP and  %
the so-called
 KPZ fixed point,  %
   the conjectural \cite{corwin2015renormalization} limiting space-time random field for the KPZ universality class which was rigorously defined on the real line by Matetski, Quastel and Remenik \cite{matetski2016kpz},
   have been among the most active areas in mathematical physics in recent decades.
 While most activities focus on the KPZ equation and KPZ fixed point on the real line (see e.g. \cite{calvert2019brownian,dauvergne2021directed,pimentel2021brownian,quastel2021convergence,quastel2017totally,Sarkar-Virag-2021,virag2020heat} and more references therein), recent progress has been made regarding the models defined on an interval instead of the real line, with appropriate boundary conditions. The investigations of an open ASEP on an interval (in a weakly asymmetric regime)  turned out to be an effective tool to study the open KPZ equation (\cite{corwin2018open,parekh2019kpz}), which then lead Corwin and Knizel   \cite{corwin21stationary} to the construction and characterization of the stationary measures of the KPZ on $[0,1]$
 (see also \cite{barraquand2022steady} and \cite{bryc21markov}).
     We refer to the references therein and to the review \cite{corwin2022some} for more background  on the topic.

The KPZ fixed point
on an interval has not yet been rigorously defined. However,
    building on the %
work of Corwin and Knizel \cite{corwin21stationary},
Barraquand and Le Doussal \cite{barraquand2022steady}
determined the large scale limit of the stationary measures of the KPZ equation under the appropriate  rescaling, and postulated that this limit  should correspond to the stationary measures of the (conjectural) KPZ fixed point on $[0,1]$ that are expected to arise as the scaling limit of
stationary measures of all models in the KPZ class on an
interval. The postulated stationary measures of the KPZ fixed point on $[0,1]$ depend on two boundary parameters $\A,\C$, %
and can be represented
as the laws of the processes
 \begin{equation}
    \label{theirH}
    \ccbb{  \wt H(x)}_{x\in[0,1]} = \ccbb{\wt B_x + \wt X_x}_{x\in[0,1]},
  \end{equation}
where $\widetilde B$ is  the standard Brownian motion multiplied by $1/\sqrt 2$ and $\widetilde X$  is an independent stochastic process with continuous trajectories.
   The  law
 $\mathbb{P}_{\wt X}$ of $\wt X$ is absolutely continuous  with respect to the law $\mathbb{P}_{\wt B}$ of process $\wt B$  with
the  Radon--Nikodym derivative $\frac{d
\proba_{\wt X}
}{d \mathbb{P}_{\wt B}}(\beta)$  proportional to
\begin{equation}\label{BD-RN}
e^{\C \min_x \beta_x+\A \min_x (\beta_x-\beta_1)}, \quad \beta = (\beta_x)_{x\in[0,1]} \in C([0,1]).
\end{equation}
 The finite dimensional distributions for $\wt X$  and their Laplace transforms were given in  %
    \cite[Supplementary material, formulas (53) and (55)]{barraquand2022steady},
  see also \cite{Bryc-Kuznetsov-2021}.
(We note change of notation here: the boundary parameters in   \cite{barraquand2022steady} are $\wt v=\A/2$  and $\wt u=\C/2$.)

  The recent advances \cite{barraquand2022steady,corwin21stationary}
leading to the process \eqref{theirH} can be interpreted as a double-limit procedure:
one first takes the limit of the increments of the height function of an ASEP in a weakly asymmetric regime as the size of the system tends to infinity (leading to stationary measure for open KPZ equation), and then scales both the boundary parameters and the magnitudes appropriately to obtain process $\wt H$. %

The contribution of this paper is a limit theorem for the open ASEP that leads to the aforementioned process \eqref{theirH}   as a single limit  when the boundary parameters change.
  The single-limit procedure seems to be of a different nature than the double-limit procedure described above. In particular the parameter $q$ is fixed. It is remarkable that $q$ does not appear in the limit process, and moreover the appearance of the process $\wt H$
  is actually a surprise to us, and we do not have a simple explanation on why the two procedures lead to the same limit process.
More specifically, in Theorem \ref{thm:1} we show that
under appropriate scaling
the
spatial height increment
process %
converges in
 finite-dimensional
distributions to the sum of two independent processes
\begin{equation}
  \label{ourH}  \frac1{\sqrt 2}\ccbb{\BB_x + \eta_x\topp{\A,\C}}_{x\in[0,1]},
\end{equation}
where   $\BB$ is the standard Brownian motion, and process $\eta\topp{\A,\C}$ is introduced in Section \ref{sec:eta}. %
 When    the sum $\A+\C$ is finite and non-negative,  convergence is in Skorokhod's space $D[0,1]$ of %
c\`adl\`ag
 functions
 and
 process \eqref{ourH} has the same law as process
 \eqref{theirH}, see Remark \ref{Rem:KPZ}.
 Our proof relies on a different representation of the limiting process, not on the Radon--Nikodym derivative \eqref{BD-RN}.  %
 Our limit theorem has actually a larger family of processes in the limit, allowing that $\A$ and/or $\C$ are infinite.

\normalsize

\subsection{Open ASEP with changing parameters}\label{sec:OASEP}
The asymmetric simple exclusion process %
is
an irreducible  continuous time
 Markov process on
the finite state space $\{0,1\}^n$ with parameters
\begin{equation}\label{eq:abcd}
\alpha>0,\quad\beta>0,\quad \gamma\geq 0,\quad \delta\geq 0,\qmand 0\leq q<1.
\end{equation}
Informally, the
process models the evolution of the particles located at sites $1,\dots,n$ that can jump to the neighbor cell to the right with rate 1 and to the left
with rate $q$, if the target site is unoccupied. Furthermore,  particles arrive at site $1$ from the left reservoir  (respectively, at site $n$ from the right reservoir),  if empty, at rate $\alpha$  (respectively, $\delta$), and
exit the system into the right reservoir at site $n$ (respectively, exit the system into the left reservoir at site $1$), if occupied, at rate $\beta$ (respectively, $\gamma$).  The transition %
rates
are summarized in Figure \ref{Fig1}.
  Since   $q<1$,
particles move in an asymmetric way,
  with higher rate to the right than to the left;
  in the special case $q = 0$,
particles move only to the right and
the model is known as the totally asymmetric simple exclusion process.

We let $\tau_1(t),\dots,\tau_n(t)$ denote the occupations of the sites: $\tau_j(t) = 1$ if the $j$-th location is occupied by a particle at time $t\geq 0$, and $\tau_j(t) = 0$ otherwise.
The height function is defined for $x\in[0,1]$ and $t\geq 0$ as
\begin{equation*}
  \label{h(x,t)}
  h_n(x,t)=h_n(0,t)+\sum_{j=1}^{\floor {n x}}(2\tau_j(t)-1), \quad h_n(0,t)=-2N_n(t),
\end{equation*}
with $N_n (t)$ denoting the net flow of particles into the system from the left reservoir up to time $t$, i.e, total number of particles that arrived at
site 1 from the left reservoir up to time $t$ minus the number of particles that have exited from
site 1 into the left reservoir up to time $t$. %
This   height function is piece-wise constant in variable $x$, which is more convenient for %
our approach than the
continuous height
function obtained  by the piece-wise linear interpolation between the jumps. %

We  denote by $\mu_n$ the stationary distribution of the ASEP as a Markov process  on $\{0,1\}^n$.
Under $\mu_n$, the process is also referred to be {\em in the steady state}
in the physics literature.  Following the common notation in the physics literature, we will denote the expectation with respect to $\mu_n$ by $\langle\cdot \rangle_n$.

Started in the steady state, i.e.,  with  $\mu_n$ as  the initial distribution of $(\tau_1(0),\dots,\tau_n(0))$, the law of the %
height increment process
$\ccbb{h_n(x,t)-h_n(0,t)}_{x\in[0,1]}$ does not %
change with
time  $t\geq 0$.  We can therefore omit the dependence on $t$ and consider
a single instance $(\tau_1,\dots,\tau_n)\in\{0,1\}^n$ as a random variable with the law $\mu_n$.
Our main object of interest is then %
 the {\em  stationary measure for the open ASEP height function process} %
   (see \cite{corwin21stationary}),
\begin{equation}\label{eq:h}
h_n(x) := \summ j1{\floor{nx}}\pp{2\tau_j-1}, \quad x\in[0,1].
\end{equation}

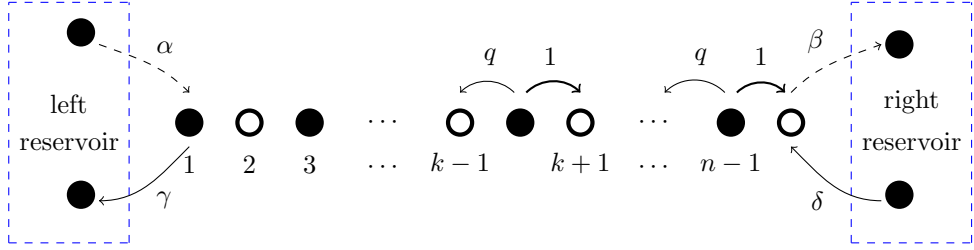
\begin{figure}[tb]
  \begin{tikzpicture}[scale=.8]
\draw [fill=black, ultra thick] (.5,1) circle [radius=0.2];
  \draw [ultra thick] (1.5,1) circle [radius=0.2];
\draw [fill=black, ultra thick] (2.5,1) circle [radius=0.2];
  \draw [ultra thick] (5,1) circle [radius=0.2];
   \draw [fill=black, ultra thick] (6,1) circle [radius=0.2];

   \node [above] at (-1.5,1.) {left};
   \node [below] at (-1.5,1.) {reservoir};
      \draw[-,dashed,blue] (-2.5,-1) to (-2.5,3);
    \draw[-,dashed,blue] (-0.5,-1) to (-0.5,3);
        \draw[-,dashed,blue] (-2.5,3) to (-0.5,3);
             \draw[-,dashed,blue] (-2.5,-1) to (-0.5,-1);
      \node [above] at (12.5,1.) {right };
       \node [below] at (12.5,1.) {reservoir};
      \draw[-,dashed,blue] (11.5,-1) to (11.5,3);
       \draw[-,dashed,blue] (11.5,3) to (13.5,3);
       \draw[-,dashed,blue] (11.5,-1) to (13.5,-1);
         \draw[-,dashed,blue] (13.5,-1) to (13.5,3);
    \draw [ultra thick] (7,1) circle [radius=0.2];
      \draw [fill=black, ultra thick] (9.5,1) circle [radius=0.2];
   \draw [ultra thick] (10.5,1) circle [radius=0.2];
     \draw[->,dashed] (-1,2.3) to [out=-20,in=135] (.5,1.5);
   \node [above right] at (-.2,2) {$\alpha$};
   \draw [fill=black, ultra thick] (-1.3,2.5) circle [radius=0.2];
     \draw[->,dashed] (10.5,1.5) to [out=45,in=200] (12,2.3);
     \node [above left] at (11.2,2) {$\beta$};
           \draw [fill=black, ultra thick] (12.3,2.3) circle [radius=0.2];
            \node  at (8.25,1) {$\cdots$};  \node  at (3.75,1) {$\cdots$};
      \node [above] at (6.5,1.8) {$1$};
      \draw[->,thick] (6.1,1.5) to [out=45,in=135] (7,1.5);
        \node [above] at (5.5,1.8) {$q$};
            \draw[<-] (5,1.5) to [out=45,in=135] (5.9,1.5);
                 \node [above] at (10,1.8) {$1$};
                \draw[->,thick] (9.6,1.5) to [out=45,in=135] (10.4,1.5);
               \node [above] at (9,1.8) {$q$};
                 \draw[<-] (8.4,1.5) to [out=45,in=135] (9.4,1.5);
       \draw[<-] (-1,-.3) to [out=0,in=-135] (.5,0.6);
   \node [below right] at (-.2,0) {$\gamma$};
     \draw [fill=black, ultra thick] (-1.3,-.2) circle [radius=0.2];
    \node [above] at (0.5,0) {$1$};
    \node [above] at (1.5,0) {$2$};
   \node [above] at (2.5,0) {$3$};
     \node [above] at (3.75,0) {$\cdots$};
   \node [above] at (5,0) {$k-1$};
       \node [above] at (7,0) {$k+1$};
        \node [above] at (8.25,0) {$\cdots$};
          \node [above] at (9.5,0) {$n-1$};
        \draw[<-] (10.5,.6) to [out=-45,in=180] (12,-.3);
   \node [below left] at (11.2,0) {$\delta$};
      \draw [fill=black, ultra thick] (12.3,-.2) circle [radius=0.2];
\end{tikzpicture}
\caption{Transition rates of the asymmetric simple exclusion process with open boundaries, with parameters $\alpha,\beta,\gamma,\delta, q$.
The black disks represent occupied sites. The white disks represent empty sites, which represent the "holes" in the discussion of the particle-hole duality.
}\label{Fig1}
\end{figure}

We shall consider the case when the parameters $\alpha,\beta,\gamma,\delta$ vary with  $n$, while we shall %
keep  $q\in[0,1)$ fixed.
  Consequently, $\mu_n$ denotes the stationary distribution of the ASEP with varying parameters $\alpha_n$, $\beta_n$, $\gamma_n$, $\delta_n$ and $q\in[0,1)$ fixed.   As in \cite{bryc19limit,bryc17asymmetric},
it is convenient to reparametrize the ASEP using
\begin{equation}\label{eq:ABCD}
A_n = \kappa_+(\beta_n,\delta_n),\; B_n = \kappa_-(\beta_n,\delta_n),\; C_n = \kappa_+(\alpha_n,\gamma_n),\; D_n = \kappa_-(\alpha_n,\gamma_n).
\end{equation}
with
\[
\kappa_{\pm}(u,v):=\frac1{2u}\pp{1-q-u+v\pm \sqrt{(1-q-u+v)^2+4uv}}.
\]
In particular $A_n,C_n\ge 0$, and %
$B_n,D_n\in(-1,0]$ as explained in \cite{bryc17asymmetric}.

We shall specify how the parameters $\alpha_n,\beta_n,\gamma_n,\delta_n$ of the ASEP vary by specifying how the parameters $A_n,B_n,C_n,D_n$ vary.
We will be interested in convergence to the "triple point" in the phase diagram, $A_n\to 1$ and $C_n\to 1$. %
  This   point   lies
 at the intersection of three  regions of the phase diagram  for the open ASEP, where the high density, low density and the maximal current regimes meet, see  Fig. \ref{Fig3}.
For parameters $B_n,D_n$ we shall consider  more general limits, with controlled rates of convergence only when $B_n\to -1, D_n\to -1$.
The key assumption on the rates of convergence is as follows.
\begin{assumption}\label{assump:0}
In addition to the   assumptions that $A_n,C_n\ge 0$, $B_n,D_n\in(-1,0]$ for all $n\geq 1$, which are implied by \eqref{eq:ABCD},  we
assume that %
$A_nC_n<1$, and
\begin{align}
\limn A_n = 1 & \mbox{ with }  \limn \sqrt n(1-A_n) = \A\in
(-\infty,\infty], \label{limAn}
\\
\limn C_n = 1& \mbox{ with }  \limn \sqrt n(1-C_n) = \C\in (-\infty,\infty], \label{limCn}
\end{align}
where $\A+\C\ge 0$.
We also assume that
\begin{equation}
    \label{limBnDn}
    \limn B_n=B\in[-1,0]  \qmand   \limn D_n=D\in[-1,0],
\end{equation}
and
\begin{equation}\label{bd:BnDn}
\limn \frac1n\log (1-B_nD_n)= 0.
\end{equation}

\end{assumption}

 We use the convention that  $\A+\C> 0$   includes the cases $\A=\infty$ and/or $\C = \infty$.

We remark that \eqref{bd:BnDn} holds if  \eqref{limBnDn} holds with $BD<1$,
   or if $BD=1$ but
the speed of convergence in one of the limits, say to $B=-1$, is restricted by a  condition such as $n^\theta (1+B_n)\to \infty$ for some $\theta> 0$.

\arxivc{ Indeed, for large enough $n$ we then have
$1+B_n\geq n^{-\theta}$ so with $D_n>-1$ and $B_n\leq 0$ we get
\[
  0\geq \tfrac1n\log (1-B_nD_n)\geq \tfrac1n\log (1+B_n)
\geq -\tfrac{\theta}n\log n\to 0.
\]

}

 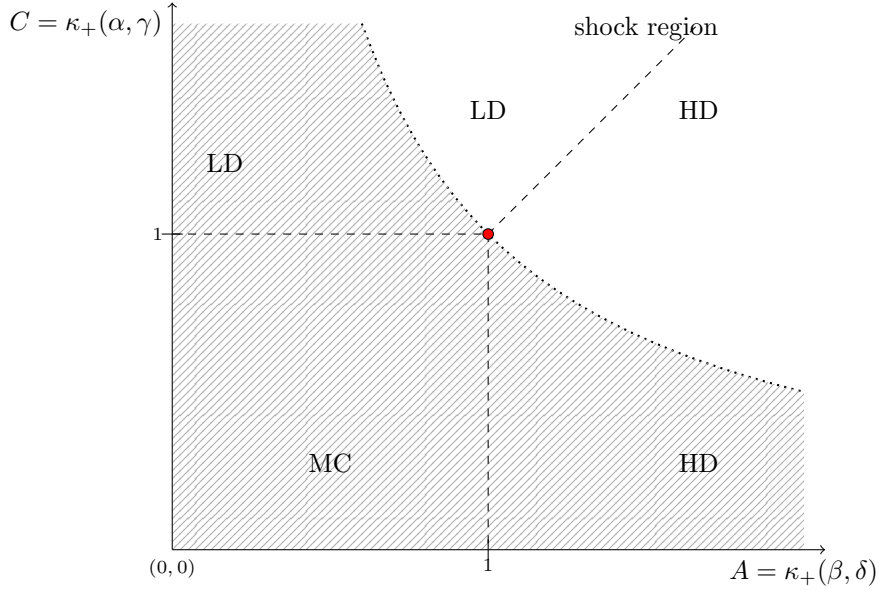
\begin{figure}[tb]

  \begin{tikzpicture}[scale=1.4]

\draw[scale = 1,domain=6.8:11,smooth,variable=\x,dotted,thick] plot ({\x},{1/((\x-7)*1/3+2/3)*3+5});

\fill[pattern=north east lines, pattern color=gray!60] (5,5)--(5,10) -- plot [domain=6.8:11]  ({\x},{1/((\x-7)*1/3+2/3)*3+5}) -- (11,5) -- cycle;

 \draw[->] (5,5) to (5,10.2);
 \draw[->] (5.,5) to (11.2,5);
   \draw[-, dashed] (5,8) to (8,8);
   \draw[-, dashed] (8,8) to (8,5);
   \draw[-, dashed] (8,8) to (10,10);
   \node [left] at (5,8) {\scriptsize$1$};
   \node[below] at (8,5) {\scriptsize $1$};
     \node [below] at (11,5) {$A = \kappa_+(\beta,\delta)$};
   \node [left] at (5,10) {$C = \kappa_+(\alpha,\gamma)$};

\node[above] at (9.5,9.75) {shock region};

  \draw[-] (8,4.9) to (8,5.1);
   \draw[-] (4.9,8) to (5.1,8);

 \node [below] at (5,5) {\scriptsize$(0,0)$};
    \node [above] at (5.5,8.5) {LD};
    \node[above] at (8,9) {LD};
    \node [below] at (10,6) {HD}; %
      \node [above] at (10,9) {HD};

 \node [below] at (6.5,6) {MC};%

\draw [fill=red] (8,8) circle [radius=0.05];
\end{tikzpicture}
\caption{
Phase        diagram for the open ASEP with maximal current (MC), low density (LD), high density (HD) regions, and with shaded  fan region $AC<1$.
Our parameters converge from within the shaded fan region to the triple point (1,1).
}\label{Fig3}
\end{figure}

\subsection{Main result}\label{sec:eta}
Let
\begin{equation}\label{eq:g}
\g_t(x,y) := \frac1{\sqrt {2\pi t}}\bb{\exp\pp{-\frac 1{2t}(x-y)^2} - \exp\pp{-\frac 1{2t}(x+y)^2}}, \quad x,y,t>0,
\end{equation}
denote the transition kernel of the Brownian motion killed at hitting zero. Introduce also
\begin{equation}\label{eq:ell}
\ell_x(y) := \frac{y}{\sqrt{2\pi x^3}}\exp\pp{-\frac{y^2}{2x}}, \quad x,y>0.
\end{equation}
We first recall two classical stochastic processes.
The {\em Brownian excursion}, denoted by $\BB^{\rm ex}$, is the process with  $\BB^{\rm ex}_0= \BB^{\rm ex}_1=0$ and with  finite-dimensional density at time points $0<x_1<\ldots<x_{d-1}<1$ of the form
\begin{equation}\label{eq:pdf_Bex}
\sqrt {8\pi} \,
\ell_{x_1}(y_1)\,\ell_{1-x_{d-1}}(y_{d-1})\,\prodd k2{d-1}\g_{x_{k}-x_{k-1}}(y_{k-1},y_k),\quad y_1,\dots,y_{d-1}>0.
\end{equation}
The {\em Brownian meander}, denoted by $\BB^{\rm me}$, is the process with $\BB^{\rm me}_0=0$ and with finite-dimensional density at time points $0<x_1<\ldots<x_{d-1}<x_d=1$ of the form
\begin{equation}\label{eq:pdf_Bme}
\sqrt{2\pi}
\ell_{x_1}(y_1)\prodd k2d\g_{x_k-x_{k-1}}(y_{k-1},y_k),\quad y_1,\dots,y_d>0.
\end{equation}
Next, we introduce  an auxiliary Markov process %
$\widetilde\eta\topp{\A,\C}=\ccbb{\widetilde\eta\topp{\A,\C}_x}_{x\in[0,1]}$
 parameterized by $\A,\C$ with $\A+\C\geq 0$.
\begin{definition} %
\begin{enumerate}[(i)]
\item The process $\widetilde \eta\topp{\infty,\infty}$ is %
the Brownian excursion $\BB^{\rm ex}$.
\item For $\A\in\RR$, the process $\widetilde\eta\topp{\A,\infty}$ is the process with  finite-dimensional density  at time points $0<x_1<\cdots<x_{d-1}<x_d = 1$ of the form
\begin{equation}
  \label{eq:proces-Ainfty}
   \frac{1}{\mathfrak C_{\A,\infty}}\,\ell_{x_1}(y_1)\,e^{-\A y_d/\sqrt 2}\,\prodd k2{d}\g_{x_k-x_{k-1}}(y_{k-1},y_k),\quad  y_1,\dots,y_d>0,
\end{equation}
where the normalization constant $\mathfrak C_{\A,\infty}$ is given in  \eqref{eq:C_ac} below.  We set $\widetilde\eta\topp{\A,\infty}_0=0$.
In particular, $\wt \eta\topp{0,\infty}$ is the Brownian meander.
\item For $\C\in\RR$, the process $\widetilde\eta\topp{\infty,\C}$ is defined as the process with finite-dimensional density  at time points $0=x_0<x_1<\cdots<x_{d-1}< 1$
 of the form
\[
 \frac{1}{\mathfrak C_{\infty,\C}}\,\ell_{1-x_{d-1}}(y_{d-1})\,e^{-\C y_0/\sqrt 2}\,\prodd k1{d-1}\g_{x_k-x_{k-1}}(y_{k-1},y_k),\quad  y_0,\dots,y_{d-1}>0,
\]
where the normalization constant  $\mathfrak C_{\infty,\C}=\mathfrak C_{\C,\infty}$ is given in \eqref{eq:C_ac} below. We set $\widetilde\eta\topp{\infty,\C}_1=0$.
\item For %
$\A+\C\in(0,\infty)$,
 the process $\widetilde\eta\topp{\A,\C}$   is defined as the process with finite-dimensional density  at time points $0=x_0<x_1<\cdots<x_d=1$ of the form
 \begin{equation}\label{eq:eta}
\wt p\topp{\A,\C}_{x_0,\dots,x_d}(y_0,\dots,y_d):= \frac1{\mathfrak C_{\A,\C}} e^{-(\C y_0+\A y_d)/{\sqrt 2}}\,\prodd k1d \g_{x_k-x_{k-1}}(y_{k-1},y_k),\quad y_0,\dots,y_d>0,
 \end{equation}
 with the expression of $\mathfrak C_{\A,\C}$ in \eqref{eq:C_ac} below.
(We note that for %
$\A+\C\in(0,\infty)$,
 process $\tfrac1{\sqrt{2}}\widetilde\eta\topp{\A,\C}$ appeared  in   \cite[Theorem 2.1]{Bryc-Kuznetsov-2021}.)
\end{enumerate}
\end{definition}

\begin{remark}
The normalization constant $\mathfrak C_{\A,\C}$ is
 \begin{equation*}
 \mathfrak C_{\A,\C}=\begin{cases}
  \int_{(0,\infty)^2} e^{-\tfrac{\C x+\A y}{\sqrt 2}}\,\g_1(x,y)\;dx\,dy, &  \mbox{ if }   \A+\C>0, \; \A,\C\in\RR,\\ \\
   \int_{(0,\infty)}\ell_1(y)e^{-\tfrac{\A y}{\sqrt 2}}\,dy, &  \mbox{ if }  \A\in\RR,\; \C=\infty,\\
 \end{cases}
 \end{equation*}
 which gives
 \begin{equation}\label{eq:C_ac}
\mathfrak C_{\A,\C}
= \begin{cases}
\displaystyle \sqrt 2\cdot \frac{\A H(\A/2)-\C H(\C/2)}{\A^2-\C^2},& \mbox{ if } \A\ne\C, \; \A+\C>0,\\\\
\displaystyle \frac{2+\A^2 }{2\sqrt 2\A}\cdot H(\A/2) - \frac1{\sqrt{2\pi}}, & \mbox{ if } \A = \C>0,
\\ \\
\displaystyle
\frac{1}{\sqrt{2\pi }}-\frac{\A H(\A/2)}{2\sqrt{2}},&  \mbox{ if } \A\in\RR,\; \C=\infty,
\\ \\
\displaystyle
\frac{1}{\sqrt{2\pi }}-\frac{\C H(\C/2)}{2\sqrt{2}},&  \mbox{ if } \A=\infty,\; \C\in\RR,
\end{cases}
\end{equation}
where for $x\in\RR$,
\begin{equation}\label{eq:H}
H(x) =e^{x^2}\erfc(x) \qmwith \erfc (x) = \frac 2{\sqrt\pi}\int_x^\infty e^{-t^2}dt.
\end{equation}
The integrals for the normalization constant can be computed with symbolic software, see also Lemma \ref{lem:C_ac}.
\end{remark}

Since $\g_t(x,y)=\g_t(y,x)$, from the form of the joint densities \eqref{eq:pdf_Bex}, \eqref{eq:pdf_Bme},
\eqref{eq:proces-Ainfty} and \eqref{eq:eta} we see that for all $\A+\C>0$, including the cases when one or both parameters are $\infty$, we have
\begin{equation}\label{eq:time-rev}
  \ccbb{\wt \eta_x\topp{\C,\A}}_{x\in[0,1]}\eqd \ccbb{\wt\eta\topp{\A,\C}_{1-x}}_{x\in[0,1]}.
\end{equation}
\arxivc{
The key is to notice that with the endpoints $x_0,x_d$ appropriately chosen for each of the four cases, the joint density satisfies
\begin{align*}
\wt p\topp{\A,\C}_{1-x_0,1-x_{d-1},\dots,1-x_1,1-x_d}(y_0,\dots,y_d)  & = \wt p\topp{\C,\A}_{x_0,x_1,\dots,x_{d-1},x_d}(y_d,\dots,y_0).
\end{align*}
For example, if $\A+\C>0$ then the joint probability density function of $\wt\eta\topp{\A,\C}$ in \eqref{eq:eta} satisfies
\begin{align*}
\wt p\topp{\A,\C}_{0,1-x_{d-1},\dots,1-x_1,1}(y_0,\dots,y_d)  & =
\frac1{\mathfrak C_{\A,\C}}e^{-\C y_0/\sqrt 2-\A y_d/\sqrt 2}\prodd k1d g_{x_k-x_{k-1}}(y_{d-k},y_{d-k+1}) \\
& = \wt p\topp{\C,\A}_{0,x_1,\dots,x_{d-1},1}(y_d,\dots,y_0).
\end{align*}
}
We can now introduce the limit stochastic processes $\eta\topp{\A,\C} = \{\eta\topp{\A,\C}_x\}_{x\in[0,1]}$ that shall arise in the scaling limit of
the stationary measure height function process
 $h_n$ in addition to a Brownian motion component.
For $\A+\C>0$, %
 we define the process $\eta\topp{\A,\C}$  as
\begin{equation}
  \label{eta:def}
  \eta\topp{\A,\C}_x:= \wt\eta\topp{\A,\C}_x - \wt\eta\topp{\A,\C}_0,\; x\in[0,1].
\end{equation}

For $\A+\C=0$, we define the process $\eta\topp{\A,-\A}$, $\A\in\R$, as
\begin{equation}\label{eq:A -A}
    \eta_x\topp{\A,-\A}:=\BB_x-\frac{\A }{\sqrt 2}x ,\; x\in[0,1].
\end{equation}

 From  definition \eqref{eq:A -A} for the case $\A+\C=0$ and from  definition
 \eqref{eta:def} combined with \eqref{eq:time-rev} for the case $\A+\C>0$, we get the following.
\begin{remark} \label{P:time-rev} For all $\A+\C\geq 0$, we have
\begin{equation}\label{eq:reversal}
\ccbb{\eta\topp{\A,\C}_x}_{x\in[0,1]}\eqd\ccbb{\eta\topp{\C,\A}_{1-x}-\eta\topp{\C,\A}_1}_{x\in[0,1]}.
\end{equation}
\end{remark}
\arxivc{ To verify \eqref{eq:reversal} when  $\C=-\A$ write,
$$\eta_{1-x}\topp{-\A,\A}-\eta_1\topp{-\A,\A}=\BB_{1-x}+\A(1-x)/2-(\BB_1+\A/2)
=(\BB_{1-x}-\BB_1)-\A x/2.$$
Since $(\BB_{1-x}-\BB_1)_{x\in[0,1]}\eqd (B_x)_{x\in[0,1]}$, the result follows.
}

In  Theorem \ref{thm:1}, which is our main result, we
establish
convergence in $D[0,1]$ for
  fluctuations of the height function  of the open ASEP in the steady state
 to the processes   predicted in \cite{barraquand2022steady}, including a slightly broader class of limits under convergence  of the finite-dimensional distributions.

 \begin{theorem}\label{thm:1}
 Under Assumption \ref{assump:0} and under the stationary distribution $\mu_n$
with $h_n$ as in \eqref{eq:h}, if $\A,\C$ are finite we have
\[
 \frac1{\sqrt n}\ccbb{h_n(x)}_{x\in[0,1]}\weakto \frac1{\sqrt 2}\ccbb{\BB_x + \eta\topp{\A,\C}_x}_{x\in[0,1]} \mbox{ as $n\to\infty$}
 \]
  as  processes in the  space $D[0,1]$ of  c\`adl\`ag functions with the Skorokhod metric and the limit process has continuous trajectories.
  Here
 $\BB$ is a standard Brownian motion, $\eta\topp{\A,\C}$ is introduced above, and the two processes are independent.

 When $\infty$ is allowed as a value for  $\A$ and/or $\C$, we still have convergence of the finite-dimensional distributions,
\begin{equation} \label{eq:MC}
\frac1{\sqrt n}\ccbb{h_n(x)}_{x\in[0,1]}\fddto \frac1{\sqrt 2}\ccbb{\BB_x + \eta\topp{\A,\C}_x}_{x\in[0,1]} \mbox{ as $n\to\infty$}.
\end{equation}

\end{theorem}

We  note that  cases   $\A=\infty$ or $\C=\infty$ in  \eqref{eq:MC}   include
Brownian excursion, Brownian meander, and its reversal  as well as  some slightly more general processes where  one of the parameters is infinite while the other one takes arbitrary real values.  Brownian excursion and  Brownian meander appeared  also in a related context %
for ASEP (but with fixed parameters)
 in \cite{bryc19limit} and in the non-rigorous discussion of the phase diagram and formula (3) in \cite{barraquand2022steady}.

 \begin{remark}\label{Rem:KPZ}
  When $\A+\C \geq 0$ are finite,  process $\widetilde X$  from \cite{barraquand2022steady}   has the same
  law as process $\frac{1}{\sqrt{2}} \eta\topp{\A,\C}$. %
Indeed,
 for $\A+\C=0$, this holds because both of these
  processes %
   are just the Brownian motion $\wt B=\BB/\sqrt{2}$, with the same drift
   $-\A x/2$.
   For $\A+\C>0$,   this can be  seen by comparing
     formula (49)  for the heat kernel, formula (53) for the joint probability function,
  and formula (57) for the normalizing constant from the supplementary material in \cite{barraquand2022steady} with the expressions \eqref{eq:g},
  \eqref{eq:eta}, \eqref{eq:C_ac} respectively.
 The Radon--Nikodym representation \eqref{BD-RN} for process $\frac{1}{\sqrt{2}}\eta\topp{\A,\C}$ is also discussed %
 in \cite{Bryc-Kuznetsov-2021}.

 \end{remark}

\arxivc{The comparison of finite dimensional laws requires  re-calculations.

If $0<x_1<x_2<\dots<x_d=1$  then from \eqref{eq:eta} the finite dimensional distribution of
$$\eta_{x_1}\topp{\A,\C},\dots,\eta_{x_d}\topp{\A,\C}$$  has density
$f(\vv y)$ in variables $y_1,\dots,y_d\geq 0$    given by
$$
f(\vv y)=\frac{e^{-\A y_d/\sqrt{2}}}{\mathfrak C_{\A,\C}}\int_0^\infty e^{-(\A+\C)z/\sqrt{2}}  \prod_{k=1}^d \g_{x_k-x_{k-1}}(z+y_{k-1},z+y_k)d z,
$$
where we set $y_0=0$.
Note that in the notation $G,G_b$ of   \cite[formula (49)]{barraquand2022steady}  we have
$$\g_t(x,y)=2^{-1/2}G(x/\sqrt{2},y/\sqrt{2},t)$$
and with $b=-z/\sqrt{2}$ we get
$$
\g_{x_k-x_{k-1}}(z+y_{k-1},z+y_k)=\frac{1}{\sqrt{2}}
G_{b}(y_{k-1}/\sqrt{2},y_k/\sqrt{2},x_k-x_{k-1})$$
After a change of variable to account for the factor $\sqrt{2}$ and
substituting $b=-z/\sqrt{2}$ in the variable of integration, we get \cite[formula (53)]{barraquand2022steady}, except for the factor $2(\wt u+\wt v)=(\C+\A)$,   recall that $2\wt u=\C$ and $2\wt v=\A$.  This additional factor is included in our normalizing constant \eqref{eq:C_ac}. In notation of  \cite[formula (57)]{barraquand2022steady} the two normalizing constants are related as follows:
\begin{equation}
    \label{C2Z}
    (\A+\C)\mathfrak C_{\A,\C}=\sqrt{2}\wt Z_{\wt u,\wt v}.
\end{equation}
}

\arxivc{
As noted in \cite{Bryc-Kuznetsov-2021}, one can also verify that for $0<\A+\C<\infty$ the law of the process $\frac{1}{\sqrt{2}}\eta\topp{\A,\C}$ is given by the Radon--Nikodym derivative  as specified in \eqref{BD-RN}, with the normalization constant $ \wt Z_{\wt u,\wt v}=(\A+\C)\mathfrak C_{\A,\C}/\sqrt{2}$.

This can be seen by combining
\cite[Theorem 1,1 and Theorem 2.1]{Bryc-Kuznetsov-2021}.
These two theorems show that for finite $\A+\C>0$, the law of the process $\frac{1}{\sqrt{2}}\eta\topp{\A,\C}$ is the limit as $\tau\to\infty$ of the process
$$\ccbb {X_x\topp \tau}_{x\in[0,1]}=\ccbb{\frac{1}{\sqrt{\tau}}X_{ x \tau }\topp{\A/\sqrt{\tau},\C/\sqrt{\tau}}}_{x\in[0,1]},$$
where $X:=\ccbb{X_x\topp{\A,\C}}_{x\in[0,\tau]}$ is the process introduced in \cite{barraquand2022steady} with the Radon--Nikodym derivative on  $C[0,\tau]$ with respect to the law of $B\topp \tau=\frac{1}{\sqrt{2}}\BB$ on $[0,\tau]$ given by
\begin{equation*}
    \label{RN-X}
        \frac{d
\proba_{ X}
}{d \mathbb{P}_{B\topp \tau}}(\beta) =\frac{1}{ \mathfrak K_{\A,\C,\tau}} e^{-\A \beta_\tau} \left(\int_0^\tau e^{-2 \beta_x}dx\right)^{-\A/2-\C/2},\;
 \beta=(\beta_x)\in C[0,\tau].
\end{equation*}
 Since process $B\topp \tau$ has the same law as process $\ccbb{\sqrt{\tau}B\topp 1 _{x/\tau}}_{x\in[0,\tau]}$, we have
\begin{equation}
    \label{RN-X-tau}
    \frac{d
\proba_{ X\topp \tau}
}{d \mathbb{P}_{B\topp 1}}(\wt \beta)
=\frac{1}{\mathfrak K_{\A/\sqrt{\tau},\C/\sqrt{\tau},\tau}}e^{-\A \wt \beta_1}\pp{ \tau \int_0^1 e^{-2 \sqrt{\tau}\wt \beta_x}dx}^{-(\A+\C)/(2\sqrt{\tau})} \; \wt \beta=(\wt \beta_x)\in C[0,1].
\end{equation}
Indeed, if       $ \frac{d
\proba_{ X}
}{d \mathbb{P}_{B\topp \tau}}(\beta)=\Phi((\beta_x)_{x\in[0,\tau]})$ for some
 $\Phi:C[0,\tau]\to\RR_+$  then $ \frac{d
\proba_{ X\topp \tau}
}{d \mathbb{P}_{B\topp 1}}(\wt \beta)=\Phi((\sqrt{\tau}\wt \beta_{x/\tau})_{x\in[0,\tau]})$ for $\wt \beta\in C[0,1]$.
With
$$\Phi(\beta)=\frac{1}{\mathfrak K}  e^{-\A \beta_\tau/\sqrt{\tau}} \left(\int_0^\tau e^{-2 \beta_x}dx\right)^{-(\A+\C)/(2\sqrt{\tau})}$$ we get
\begin{multline*}
  \Phi((\sqrt{\tau}\wt \beta_{x/\tau})_{x\in[0,\tau]})
=\frac{1}{\mathfrak K}  e^{-\A (\sqrt{\tau} \wt \beta_1)/\sqrt{\tau}} \left(\int_0^\tau e^{-2 \sqrt{\tau}\wt \beta_{x/\tau}}dx\right)^{-(\A+\C)/(2\sqrt{\tau})}
\\=\frac{1}{\mathfrak K}e^{-\A   \wt \beta_1 } \left(\tau\int_0^1 e^{-2 \sqrt{\tau}\wt \beta_{x}}dx\right)^{-(\A+\C)/(2\sqrt{\tau})}
\end{multline*}

We now check that the density \eqref{RN-X-tau} converges to the correct limit. We first verify convergence of the normalization constant.
Assuming for simplicity that $\A\ne \C$  from \cite[Formula (1.14) and Lemma 4.4]{Bryc-Kuznetsov-2021} we get
$$\lim_{\tau\to\infty} \mathfrak K_{\A/\sqrt{\tau},\C/\sqrt{\tau},\tau}= \frac{\A H(\A/2)-\C H(\C/2)}{\A-\C}.
$$
which differs by just a factor of $\sqrt{2}/(\A+\C)$ from $\mathfrak C_{\A,\C}$ as defined in \eqref{eq:C_ac}. In view of \eqref{C2Z}, this recovers the normalizing constant $\wt Z_{\wt u,\wt v}$.
 Since $\A+\C>0$,
$$\lim_{\tau\to\infty} \pp{\tau \int_0^1 e^{-2 \sqrt{\tau}\beta_x}dx}^{-(\A+\C)/(2\sqrt{\tau})}
=e^{(\A+\C)\min_{0\leq x\leq 1}\beta_x}.$$
Thus the Radon--Nikodym density \eqref{RN-X-tau} converges to the expression \eqref{BD-RN} with the normalization constant $\wt Z_{\wt u,\wt v}$ as  required.
}

\subsection{Organization of the paper} In Section \ref{sec:AW} we review Askey--Wilson processes, their relation to  the  matrix product ansatz for the open ASEP,  the particle-hole duality,
and a coupling technique which we use in Section \ref{sec:a+c=0} to prove Theorem \ref{thm:1} for  $\A+\C=0$.

In Section \ref{sec:a+c>0} we prove Theorem \ref{thm:1}   in  the case     $\A+\C>0$.
  The proof consists of two main steps.
  First, in Theorem \ref{thm:2} we compute  the limit of the Laplace transform
\[
\aa{\exp\pp{-\frac1{\sqrt n}\summ k1d c_kh_n(x_k)}}_n, \quad  0<x_1<\cdots<x_d=1,
\]
as $n\to\infty$.
Second,  in Section \ref{sec:duality} we
identify the limit as the Laplace transform of the process $(\BB+\eta\topp{\A,\C})/\sqrt{2}$.
In the proof, we work
under the additional assumption that $A_n\ge C_n$ (thus $\A\leq\C$), and the result for $A_n<C_n$ follows by the particle-hole duality.

At the heart of our computation of the limit Laplace transform, we rely on the  matrix ansatz of Derrida, Evans,  Hakim and Pasquier \cite{derrida93exact} and its representation via Askey--Wilson Markov processes, which are reviewed in Sections \ref{Sec:AWP} and \ref{Sec:MS4ASEP}.
Our approach is the same as the one used in recent developments on limit theorems for the height increment process of the open ASEP with fixed parameters  \cite{bryc19limit} and for open
weakly asymmetric simple exclusion process
 \cite{corwin21stationary}. The analysis here is more involved than the one in \cite{bryc19limit}, but much less so than the one in  \cite{corwin21stationary} where all the parameters depend on the size of the system.
\section{Preliminaries}
\label{sec:AW}

\subsection{Askey--Wilson processes}\label{Sec:AWP}

The Askey--Wilson polynomials were    introduced by Askey and Wilson  \cite{askey85some}.
The polynomials satisfy a three step recursion, so when Favard's theorem applies the polynomials are orthogonal with respect to a probability measure on the real line, which we shall call  the Askey--Wilson measure.
The Askey--Wilson processes are a family of Markov processes based on  these Askey--Wilson measures. The  material of this section is mainly based on \cite{bryc19limit}, which in turn is based  mostly on \cite{askey85some,bryc10askey}.
The Askey--Wilson probability measure $\nu(dy;a,b,c,d,q)$ depends on %
 five
parameters $a,b,c,d,q$ and it is invariant with respect to permutations of %
$a,b,c,d$.
It is assumed that $q\in(-1,1)$. %
For the   parameters $a,b,c,d$, it is assumed that they are all real, or two of the parameters are real and the other two form a complex conjugate pair,
or the parameters form two complex conjugate pairs.
It is difficult to  give general conditions on the parameters  that ensure
existence of the Askey-Wilson measure $\nu(dy;a,b,c,d,q)$.
A simple sufficient condition is
\begin{equation}\label{eq:restriction}
ac,ad,bc,bd,qac,qad,qbc,qbd, abcd, qabcd \not \in [1,\infty).
\end{equation}
Conditions that allow some of the above products to be in $[1,\infty)$ are listed in \cite[Lemma 3.1]{bryc10askey}.
In this paper we will  encounter  the degenerate Askey--Wilson measure corresponding to \cite[Lemma 3.1(iii)]{bryc10askey} with $N=0$.

In general, the  Askey--Wilson measure is   of mixed type
\[
\nu(dy;a,b,c,d,q)=f(y;a,b,c,d,q)dy+\sum_{z\in F(a,b,c,d,q)}\,p(z)\delta_z(dy),
\]
with the absolutely continuous part supported on $[-1,1]$ and with the discrete part supported on a finite or empty set $F$.
For certain choices of parameters, the measure can be only discrete or only absolutely continuous.
The absolutely continuous part is %
\begin{equation}\label{eq:f}
  f(y;a,b,c,d,q)=\frac{\qp{q,ab,ac,ad,bc,bd,cd}}{2\pi\qp{abcd}\sqrt{1-y^2}}\,\left|\frac{\qp{e^{2i\theta_y}}}
{\qp{ae^{i\theta_y},be^{i\theta_y},ce^{i\theta_y},de^{i\theta_y}}}\right|^2,
\end{equation}
where $y=\cos\,\theta_y$
(with the convention that $f(y;a,b,c,d,q)=0$ when $|y|>1$).
Here and below, for complex
$\alpha,\alpha_1,\dots,\alpha_k$,  $n\in\N\cup\{\infty\}$ and $|q|<1$ we
use the
$q$-Pochhammer symbol
\begin{equation*}
\qps{\alpha}_n=\prod_{j=0}^{n-1}\,(1-\alpha q^j), \quad \qps{\alpha_1,\cdots,\alpha_k}_n =\prodd j1k\qps{\alpha_j}_n.
\end{equation*}
It will be convenient to use the identity $\qp{a}=(1-a)\qp{q a}$ to separate the first factors  in \eqref{eq:f}  from  the remaining infinite products, so we write
\begin{equation*}
  f(y;a,b,c,d,q)=\frac{2}{\pi}J(y;a,b,c,d)R(y;a,b,c,d,q)
\end{equation*}
with
\begin{equation}\label{eq:J}
  J(y;a,b,c,d)=\frac{(1-ab)(1-ac)(1-ad)(1-bc)(1-bd)(1-cd) \sqrt{1-y^2}}{(1-abcd)
  \left|(1-ae^{i\theta_y})(1-be^{i\theta_y})(1-ce^{i\theta_y})(1-de^{i\theta_y})\right|^2}
, %
\end{equation}
\begin{equation}\label{eq:R}
  R(y;a,b,c,d,q)=\frac{\qp{q,qab,qac,qad,qbc,qbd,qcd} \left|\qp{qe^{2i\theta_y}}\right|^2}
  {\qp{qabcd}\left|\qp{qae^{i\theta_y},qbe^{i\theta_y},qce^{i\theta_y},qde^{i\theta_y}}\right|^2}.
\end{equation}

The set $F=F(a,b,c,d,q)$  of atoms of $\nu(dy;a,b,c,d,q)$ is   non-empty if  there is a real parameter $\alpha\in\{a,b,c,d\}$  with
$|\alpha|> 1$. Each such parameter generates atoms.
For example,
 if $|a|> 1$ then it generates the atoms
\begin{equation*}
  y_j=\frac12\pp{aq^j+\frac1{aq^j}} \mbox{ for $j=0,1,\dots$ such that $|aq^j|\ge 1$},
\end{equation*}
and the corresponding probabilities are
\begin{align*}
 p(y_0;a,b,c,d,q) & =\frac{\qp{a^{-2},bc,bd,cd}}{\qp{b/a,c/a,d/a,abcd}},\; \label{eq:p0}
 \\
p(y_j;a,b,c,d,q) & =p(y_0;a,b,c,d,q)\frac{\qps{a^2,ab,ac,ad}_j\,(1-a^2q^{2j})}{\qps{q,qa/b,qa/c,qa/d}_j(1-a^2)}\left(\frac{q}{abcd}\right)^j,\;
j\ge 1.\nonumber
\end{align*}
The %
expression for
$p(y_j;a,b,c,d,q)$  given here only applies for $a,b,c,d\ne 0$, and takes a different form otherwise. We  shall however  only
need $p(y_0;a,b,c,d,q)$ in this paper (with  $a=0$ defined as  the limit $a\to 0$).

The Askey--Wilson process is a time-inhomogeneous Markov process introduced in \cite{bryc10askey}, based on the Askey--Wilson measures. It is then explained in~\cite{bryc17asymmetric} how each ASEP with parameters $\alpha,\beta>0, \gamma,\delta\ge 0, q\in[0,1)$ is associated to an Askey--Wilson process $Y$, the parameters of which are denoted by $A,B,C,D,q$, with $A,B,C,D$ given in~\eqref{eq:ABCD}.

As we already noted, \eqref{eq:ABCD} implies $A,C\geq 0$ and $-1<B,D\leq 0$.
  Throughout the paper we shall assume  $AC<1$, which ensures that condition \eqref{eq:restriction} holds for the marginal laws (but not for the transition probabilities).
Then, the Askey--Wilson process with parameters $(A,B,C,D,q)$ is introduced as the Markov process with marginal  distribution
\begin{equation}\label{eq:AW marginal}
\proba(Y_t\in dy)=\nu\pp{dy;A\sqrt{t},B\sqrt t,C/\sqrt{t},D/\sqrt t,q},\quad 0<t<\infty,
\end{equation}  and the transition probabilities %
\begin{equation}
\label{eq:AW transition}
\proba(Y_t\in dz\mid Y_s=y)=\nu\pp{dz;A\sqrt{t},B\sqrt t,\sqrt{s/t}(y+\sqrt{y^2-1}),\sqrt{s/t}(y-\sqrt{y^2-1})},
\end{equation}
for $0<s<t$, $y,z>0$.
When $|y|<1$, expression $y\pm \sqrt{y^2-1}$ is understood as $e^{\pm i\theta_y}$ with  $\cos\theta_y = y$. It was shown in \cite{bryc10askey} that the above marginal laws and  transition probabilities satisfy the Chapman-Kolmogorov equations and hence determine a Markov process indexed by $t\in[0,\infty)$.
\arxivc{To recognize when the transition probabilities are degenerate, we will use the formula for the conditional variance. With $0< s<t$,  \cite[Proposition 2.5]{bryc10askey} gives
\begin{equation}\label{condmom1}
\esp(Y_t|Y_s)=\frac{(A+B)(t-s)+2(1-ABt)\sqrt{s}Y_s}{2\sqrt{t}(1-ABs)}\;,
\end{equation}
\begin{equation*}\label{condvar}
{\rm{Var}}(Y_t|Y_s)
=\frac{(1-q)(t-s)(1-ABt)}{4t(1-ABs)^2(1-qABs)}(1+A^2s-2A\sqrt{s}Y_s)(1+B^2s-2B\sqrt{s}Y_s)\;.
\end{equation*}
Strictly speaking, \cite[Proposition 2.5]{bryc10askey} was proved in the absolutely continuous case. However, the same formulas hold in more generality: formula \eqref{condmom1} just specifies  the root of the first Askey--Wilson polynomial defined in \cite[(3.12)]{bryc10askey}. The conditional variance can  be read out from the explicit expression for the second orthogonal polynomial.
The Askey--Wilson process turned out to be closely related  to a large family of Markov processes,  the so-called quadratic harnesses \cite{bryc07quadratic} in the literature; see \cite[Section 1.3]{bryc17asymmetric} for more on this connection.}

More explicit expressions for the law of $Y$ will  appear below when %
needed in the proofs.
\subsection{Matrix ansatz  for open ASEP and Askey--Wilson processes}
\label{Sec:MS4ASEP}
Recall that $\aa \cdot _n$ denotes the expectation with respect to
the invariant measure  $\mu_n$ of the open ASEP.
\cite{derrida93exact} introduced the celebrated %
matrix product  ansatz
that
 provides    an explicit expression for the joint generating function.
 Formally,
  for any $t_1,\dots,t_n>0$,
 from \cite{derrida93exact}
 one can write %
\begin{equation*}\label{MatrixAnsatz}
\aa{\prodd j1n t_j^{\tau_j}}_n = \frac{\langle W|(\mathsf{E}+t_1\mathsf{D})\times\cdots\times(\mathsf{E}+t_n\mathsf{D})|V\rangle}{\langle W|(\mathsf{E}+\mathsf{D})^n|V\rangle},
\end{equation*}
for a pair of infinite matrices $\mathsf{D},\mathsf{E}$, a row vector $\langle W|$ and a column vector $|V\rangle$, satisfying
\begin{align*}
\mathsf {DE} - q\mathsf{ED} & = \mathsf {D} + \mathsf {E},\\
\langle W|(\alpha \mathsf E - \gamma \mathsf D) & = \langle W|,\\
(\beta\mathsf D - \delta\mathsf E)|V\rangle & = |V\rangle.
\end{align*}
See \cite{derrida06matrix,derrida07nonequilibrium}   for  reviews of the literature.
See also \cite{Bryc-Swieca-2018}, in particular Appendix C there, for a discussion of the case where the matrix approach fails.   However, for our purpose, we shall apply an alternative expression developed recently
in \cite[Theorem 1]{bryc17asymmetric}, summarized
in
 the following theorem.
\begin{theorem}\label{T-BW17}
Consider the parametrization $A,B,C,D$ in~\eqref{eq:ABCD} for an open ASEP with parameters $\alpha,\beta>0,\gamma,\delta\ge 0$ and $q\in[0,1)$.
Suppose that $AC<1$.

If $0<t_1\leq t_2\leq \dots\leq t_n$, then the joint generating function of
the stationary distribution $\mu_n$ of the ASEP %
is
  \begin{equation}
 \label{eq:BW}
\aa{\prod_{j=1}^n t_j^{\tau_j}}_n=\frac{\esp\bb{\prod_{j=1}^n(1+t_j+2\sqrt{t_j}\,Y_{t_j})}}{2^n\esp(1+Y_1)^n},
  \end{equation}
  where $\{Y_t\}_{t\ge 0}$ is the Askey--Wilson process with parameters $(A,B,C,D,q)$.
\end{theorem}

\subsection{Particle-hole duality}\label{sec:PH}
The asymptotics for the height function in the low density phase $A_n\leq C_n$ is an immediate consequence of the asymptotics for the high density phase $A_n\geq C_n$, by the particle-hole duality which we now explain.

Consider an open ASEP of size $n$ with parameters
$(\alpha,\beta,\gamma,\delta,q)$ and stationary law that is convenient here to denote by $\mu_n^{\alpha,\beta,\gamma,\delta}$.
 (Since $q$ is fixed throughout the paper, we suppress the dependence on $q$.)
Instead of thinking of particles jumping around, we can view the particles as background and allow the holes to jump around. %
In this way, equivalently a hole jumps to the unoccupied left and right sites with rates 1 and $q$, respectively, and disappears at site $1$ with rate $\alpha$ and
at
site $n$ with rate $\delta$, and enters site $n$ if unoccupied with rate $\beta$ and site $1$ if unoccupied with rate
$\gamma$. The holes form   the open ASEP with parameters
$(\wt \alpha,\wt \beta,\wt \gamma,\wt \delta,q)=(\beta,\alpha,\delta,\gamma,q)$,
 if we relabel the sites $\{1,\dots,n\}$ as $\{n,\dots,1\}$ by $j\mapsto n-j+1$.
Consequently, the particles occupations $\tau_1,\dots,\tau_n$, represented by the black disks in Fig. \ref{Fig1}, are related to the holes occupations $\varepsilon_1,\dots,\varepsilon_n$, represented by the white disks in Fig. \ref{Fig1}, by
\begin{equation}\label{*eps*}
    \varepsilon_j = 1-\tau_{n-j+1},\quad 1\leq j\leq n,
\end{equation}
and the stationary law of $(\eps_1,\dots,\eps_n)$ is $\mu_n^{\beta,\alpha,\delta,\gamma}$.
  Introduce
\begin{equation*}
    \label{*wt h*}
    \what h_n(x) = \summ j1{\floor {nx}}\pp{2\varepsilon_j-1}.
\end{equation*}

The above argument shows that $\{\what h_n(x)\}_{x\in[0,1]}$ with respect to
$\mu_n^{\beta,\alpha,\delta,\gamma}$
has the same law as $\{h_n(x)\}_{x\in[0,1]}$ (defined in~\eqref{eq:h}) with respect to $\mu_n^{\alpha,\beta,\gamma,\delta}$.
This allows us to switch the roles of the pairs of parameters $(A,B)$ and $(C,D)$ in \eqref{eq:ABCD} which will simplify the proof of Theorem \ref{thm:1}.

\begin{proposition}
  \label{C:swap}
  If Theorem \ref{thm:1} holds  under an additional assumption that $A_n\geq C_n$ for all $n$, then it holds also without this additional assumption.
\end{proposition}
\begin{proof}
Suppose that a sequence  $(A_n,B_n,C_n,D_n)$  satisfies Assumption \ref{assump:0}. Consider sets of indexes
$\NN_+=\{n\in\NN: A_n\geq C_n\}$ and $\NN_-=\NN\setminus \NN_+$, at least one of which must be infinite.

 If $\NN_+$ is infinite,  we can extend $(A_n,B_n,C_n,D_n)_{n\in\NN_+}$
to a sequence $(\wt A_n,\wt B_n,\wt C_n, \wt D_n)_{n\in\NN}$ such that $\wt{A}_n\geq \wt C_n$ for all $n$. (For example, if $\NN_+=\{n_1,n_2,\dots\}$,  with $n_0:=0$
we can take $\wt A_n=A_{n_k}$ for $n\in[n_{k-1}+1,n_k]$.)
So our assumption implies that the limit \eqref{eq:MC} holds over the sub-sequence $\NN_+$.

It remains to prove that if $\NN_-$ is an infinite set, then we have convergence in  \eqref{eq:MC}  over $\NN_-$ to the same limit.
With some abuse of notation, let us re-parameterize the stationary measure of the open ASEP by the parameters $A_n,B_n,C_n,D_n$, writing $\mu_n^{A_n,B_n,C_n,D_n}$ instead of $\mu_n^{\alpha_n,\beta_n,\gamma_n,\delta_n}$.
 The stationary measure for the hole occupations is then  $\mu_n^{\beta_n,\alpha_n,\delta_n,\gamma_n}$ which, using \eqref{eq:ABCD}, we write as $\mu_n^{C_n,D_n,A_n,B_n}$.
 \arxivc{
Note that with $\wt \alpha_n=\beta_n$, $\wt \beta_n=\alpha_n$,
$\wt \gamma_n=\delta_n$ and $\wt \delta_n=\gamma_n$, formula \eqref{eq:ABCD}
applied to the parameters with the tilde gives
$\wt A_n=\kappa_+(\wt \beta_n,\wt \delta_n)=\kappa_+(\alpha_n,\gamma_n)=C_n$.
}
For $n\in\NN_-$, we have $C_n>A_n$, so  our assumption implies that Theorem \ref{thm:1} holds for $\what h_n$, with the limit taken  over  $\NN_-$. We get
\[
\frac1{\sqrt n}\ccbb{\what h_n(x)}_{x\in[0,1]}\fddto \frac1{\sqrt 2}\ccbb{\BB_x + \eta_x\topp{\C,\A}}_{x\in[0,1]}.
\]
 Observe that \eqref{*eps*} gives
\begin{equation}\label{eq:particle_hole}
h_n(x) - \pp{\what h_n(1-x)-\what h_n(1)} =
\displaystyle (1-2\tau_{\ceil{ nx}})\,  \ind_{nx\ne \ceil{nx}}.
\end{equation}
\arxivc{
Indeed, noting that $\floor{n - nx}=n-\ceil{nx}$ for the last line, the left hand side of  \eqref{eq:particle_hole} is
\begin{multline*}
 h_n(x) + \pp{\what h_n(1)-\what h_n(1-x)}=   \sum_{j=1}^{\floor{nx}}(2\tau_j-1)+\sum_{j=\floor{(1-x)n}+1}^n (2\eps_j-1)
    \\=\sum_{j=1}^{\floor{nx}}(2\tau_j-1) +\sum_{j=\floor{(1-x)n}+1}^n (1-2\tau_{n-j+1})
    \\=\sum_{j=1}^{\floor{nx}}(2\tau_j-1) +\sum_{k=1}^{n-\floor{(1-x)n}} (1-2\tau_{k})
        \\=\sum_{j=1}^{\floor{nx}}(2\tau_j-1) +\sum_{k=1}^{\ceil{nx}} (1-2\tau_{k}) =(1-2\tau_{\ceil{ nx}})\ind_{nx\ne \floor{nx}}.
\end{multline*}
}
Since the difference \eqref{eq:particle_hole}  is uniformly bounded,
the finite-dimensional distributions of %
$n^{-1/2}\sccbb{h_n(x)}_{x\in[0,1]}$ have the same limit
as the finite-dimensional distributions of %
$n^{-1/2}\sccbb{\what h_n(1-x) - \what h_n(1)}_{x\in[0,1]}$, and
 we arrive at
\[
\frac1{\sqrt n}\ccbb{h_n(x)}_{x\in[0,1]} \fddto  %
\frac1{\sqrt 2} \ccbb{\BB_{1-x} - \BB_{1} + \eta_{1-x}\topp{\C,\A}-\eta_1\topp{\C,\A}}_{x\in[0,1]} \eqd \frac1{\sqrt 2}\ccbb{\BB_x + \eta_x\topp{\A,\C}}_{x\in[0,1]},
\]
where in the last equality  we used
  $\ccbb{\BB_{1-x} - \BB_{1}}_{x\in[0,1]} \eqd \ccbb{\BB_x}_{x\in[0,1]}$ and \eqref{eq:reversal} from Remark \ref{P:time-rev}. Since $\NN=\NN_+\cup \NN_-$,  this ends the proof.
\end{proof}

\subsection{Coupling and tightness in $D[0,1]$ }\label{Sect:Coupling}
As in \cite{corwin21stationary}, we will deduce tightness  by coupling a  realization of the simple exclusion processes in the steady state  with  two sequences of $\{0,1\}$ valued random variables that have  products of  Bernoulli measures as the  marginal laws.

As in Section \ref{sec:PH}, we denote by $\mu_n^{A_n,B_n,C_n,D_n}$ the stationary distribution of the ASEP on $\{1,\dots,n\}$
with parameters $(q,A_n,B_n,C_n,D_n)$, where $q$ remains fixed.

Since we are interested only in the case $A_n\to 1$ and  $C_n\to 1$, without loss of generality we assume that parameters $A_n,C_n$ are not zero for all $n$.

Let $(\tau_1,\dots \tau_n)$ be a vector with the stationary law $\mu_n^{A_n,B_n,C_n,D_n}$. Define $\what A_n=1/C_n$ and $\wt C_n=1/A_n$. Let   $(\wt \tau_1,\dots \wt \tau_n)$  be a vector with the stationary law $\mu_n^{ A_n,B_n,\wt C_n,D_n}$
 (just one parameter changed),  and
let $(\what \tau_1,\dots, \what \tau_n)$ be a vector with the stationary law $\mu_n^{\what A_n,B_n,C_n,D_n}$.

From \cite[Lemma 5.1]{corwin21stationary} we then deduce the following result. (A  similar coupling for a pair of ASEPs is constructed in \cite[Lemma 2.1]{gantert2020mixing}.)

\begin{proposition} \label{thm-coupling} \begin{enumerate}[(i)]
  \item Random variables  $(\what \tau_1,\dots,\what \tau_n)$ are independent Bernoulli   random variables with $\proba(\what \tau_j=1)=\frac{1}{1+C_n}$, and random variables  $(\wt \tau_1,\dots, \wt \tau_n)$ are
  independent Bernoulli random variables with $\proba(\wt \tau_j=1)=\frac{ A_n}{1+A_n}$.
  \item The three vectors can be defined together on a single probability space in such a way that
\begin{equation}
  \label{compare-tau}
  \wt \tau_j\leq \tau_j\leq \what \tau_j,\quad j=1,2,\dots,n.
\end{equation}
\end{enumerate}
\end{proposition}

\begin{proof}
Since $\what A_nC_n=1$ it is clear that random variables $(\what \tau_1,\dots, \what \tau_n)$ are Bernoulli with $\proba(\what \tau_j=1)=\frac{\what A_n}{1+\what A_n}=\frac{1}{1+C_n}$, see
\cite[Remark 2.4]{bryc19limit}. Similarly, random variables $(\wt\tau_1,\dots,\wt \tau_n)$ are Bernoulli with $\proba(\wt\tau_j=1)=\frac{ A_n}{1+A_n}$. This proves statement (i).

  To construct the coupling of the three vectors, recall that ASEP parameters for $(\tau_1,\dots \tau_n)$ are given by the inverse of formula \eqref{eq:ABCD}, which gives
    \begin{align*}
  \alpha_n & = \frac{1-q}{(1+C_n)
   (1+{D_n})}, &
 \beta_n& =\frac{1-q}{(1+A_n)(1+B_n)},\\
   \gamma_n & = \frac{-(1-q)C_nD_n}{(1+C_n)(1+D_n)}, &\delta_n& = \frac{-(1-q)A_nB_n}{(1+A_n)(1+B_n)}.
   \end{align*}
   Recall that under our assumption we have $A_n,C_n> 0$,  $1+B_n,1+D_n>0$, and $-B_n,-D_n\geq 0$  and also $1-q>0$. Since $\what A_n=1/C_n>A_n$, therefore, the
     ASEP parameters for $( \what \tau_1,\dots, \what \tau_n)$ are $\what \alpha_n=\alpha_n, \what \beta_n\leq \beta_n,\what \gamma_n=\gamma_n,\what \delta_n\geq \delta_n $.
   Similarly, since $\wt C_n>C_n$ the ASEP parameters for $(\wt \tau_1,\dots \wt \tau_n)$ are  $\wt \alpha_n\leq \alpha_n$, $\wt \beta_n=\beta_n$, $\wt \gamma_n\geq \gamma_n$, $\wt \delta_n=\delta_n$.
Invoking \cite[Lemma 5.1]{corwin21stationary}, we get  \eqref{compare-tau}.  %
\end{proof}
For reader's convenience we describe the construction from \cite[Lemma 5.1]{corwin21stationary}, adapted to our setting and notation.
\begin{proof}[Sketch of proof of Lemma 5.1 in \cite{corwin21stationary} specialized to the case \eqref{compare-tau}]
Consider a three-species ASEP, with three types of particles:  the high-priority red, mid-priority blue, low-priority gray particles and with   the empty slots  that can be interpreted as an additional type of particle with the  lowest priority.  At any time $t\geq 0$, a site can be occupied  by at most one particle.  The time evolution of the three-species ASEP may  start from any initial state, which for concreteness we take  to be an empty state and is  a Markov process with the $4^n$-element state space $\{red,blue,gray,empty\}^n$.

We first define the transition rates "away from the boundaries". (This is of course a somewhat informal description, see \cite[Lemma 2.1]{gantert2020mixing} for a more precise approach.)
 \begin{enumerate}
   \item A red particle may jump to the right from site $j=1,\dots,n-1$  or  to the left at rate $q$ from site $j=2,\dots,n$, provided that the target site does not have a red particle. When a  red particle moves to an empty site or a site occupied by a blue, or gray particle, the particles swap their locations.

      \item  A blue particle may   jump to the right at rate 1 and to the left at rate $q$, provided that the target site does not have a red or a blue particle. When a blue particle moves to the empty site or a site occupied by a   gray particle,  the particles swap the locations.

          \item A gray  particle may jump to the right at rate 1 or to the left at rate $q$, provided that the target site is empty.

 \end{enumerate}
The evolution at the end points is described by the "mutation rates", where  a particle, including an empty site, can "mutate" into one of the other species.
The mutation rates are specified in the two  tables, one for each of the end-point locations.  The tables list  the transition (mutation) rates from a color specified in the left column into one of the colors listed in the first row.
For example, at the left endpoint, transitions gray$\mapsto$ blue occur at rate $\alpha_n-\wt \alpha_n$.

\medskip
 \hspace{-.5cm} \begin{tabular}{ccc}
  Left endpoint transition rates: & \hfill & Right endpoint transition rates: \\
\begin{tabular}{c|cccc}
&empty &gray& blue & red \\ \hline
empty &  0&0 & $\alpha_n-\wt \alpha_n$ &$\wt \alpha_n$\\\\
gray &$\gamma_n$ & 0 & $\alpha_n-\wt \alpha_n$ &$\wt \alpha_n$\\\\
blue &$\gamma_n$ &0 & 0 &$\wt \alpha_n$\\\\
red &$\gamma_n$ & 0&$\wt\gamma_n-\gamma_n$ & 0 \\
 \end{tabular}
&\hfill &
\begin{tabular}{c|cccc}
&empty &gray& blue & red \\ \hline
empty & 0& $\what \delta_n-\delta_n$&0  &$\delta_n$ \\ \\
gray &$\what \beta_n$ & 0 &0 &$\delta_n$\\\\
blue &$\what \beta_n$ &$\beta_n-\what \beta_n$ & 0 & $\delta_n$\\\\
red & $\what \beta_n$&$\beta_n-\what\beta_n$ & 0& 0 \\
 \end{tabular}
 \\
 \end{tabular}

 \medskip
The above multi-species ASEP evolution defines the following three vectors in $\{0,1\}^n$: %
 \begin{itemize}
   \item  $\wt \tau_j(t)=1$ when at time $t$ the $j$-th location is occupied by a   red particle
      \item $\tau_j(t)=1$ when at time $t$ the $j$-th location is occupied either by a red or a blue  particle
    \item  $\what \tau_j(t)=1$ when at time $t$ the $j$-th location is non-empty.
 \end{itemize}
 It is clear that
 \begin{equation*}
  \label{compare-tau(t)}
  \wt \tau_j(t)\leq \tau_j(t)\leq \what \tau_j(t),\quad j=1,2,\dots,n.
\end{equation*}
It is also clear that the evolution of vector $(\tau_1(t),\dots,\tau_n(t))$ alone (marginal law) is an ASEP with parameters $(q, A_n,B_n,C_n,D_n)$ and similarly, the other two vectors are the ASEPs with the parameters  $(q, A_n,B_n,\wt C_n,D_n)$ and  $(q, \what A_n,B_n,C_n,D_n)$ respectively.
 To conclude the construction, we define the joint law of $(\tau_1,\dots,\tau_n,\wt \tau_1,\dots,\wt \tau_n, \what \tau_1,\dots,\what \tau_n)$ as the stationary law of the Markov process, i.e., as the weak limit
 \[
 (\tau_1(t),\dots,\tau_n(t),\wt \tau_1(t),\dots,\wt \tau_n(t), \what \tau_1(t),\dots,\what \tau_n(t))\Rightarrow (\tau_1,\dots,\tau_n,\wt \tau_1,\dots,\wt \tau_n, \what \tau_1,\dots,\what \tau_n)
 \]
  as $t\to\infty$.
 The limit defines the appropriate joint law   such that \eqref{compare-tau} holds, and the marginal laws are the stationary laws of the corresponding ASEPs. (The first ASEP consists of only red particles. The second ASEP consists of   red or blue particles. The third ASEP is color-blind.)
 \end{proof}
 A useful technical consequence of Proposition \ref{thm-coupling} is tightness.
\begin{proposition}
  \label{Prop-tightness} If Assumption \ref{assump:0} holds with finite $\A,\C$, then the laws of
  \[
  \frac{1}{\sqrt{n}} \ccbb{  h_n(x)}_{x\in[0,1]}
  \]
  are tight in $D[0,1]$ and their subsequential limits have continuous trajectories.
\end{proposition}
\begin{proof}
   In the context of Proposition  \ref{thm-coupling}, in addition to $h_n$ defined in \eqref{eq:h},  consider
 the height function processes for  the two new ASEPs introduced there.
 It will be   convenient to consider the "high density centering"   and the "low density" centering as in \cite[(1.4)]{bryc19limit}:
\begin{equation}\label{eq:hhh}
\wt h_n^{\rm H}(x) := \summ j1{\floor{nx}}\pp{2\wt \tau_j-\frac{2A_n}{1+A_n}}, \quad x\in[0,1].
\end{equation}
\begin{equation}\label{eq:hh}
\what h_n^{\rm L}(x) := \summ j1{\floor{nx}}\pp{2\what \tau_j-\frac{2}{1+C_n}}, \quad x\in[0,1].
\end{equation}
Also introduce two deterministic functions:
\begin{equation}\label{the:eps}
  \wt \eps_n(x):= \floor{nx}\frac{A_n-1}{1+A_n} , \quad \mbox{ and } \quad \what \eps_n(x):= \floor{nx}\frac{1-C_n}{1+C_n}.
\end{equation}

Denoting $\Delta_{x_0,x_1}(f):=f(x_1)-f(x_0)$,
from  Proposition \ref{thm-coupling}  we get
\begin{equation}
  \label{h-sandwitch}
 \Delta_{x_0,x_1}( \wt h_n^{\rm H})+  \Delta_{x_0,x_1}(\wt \eps_n) \leq \Delta_{x_0,x_1}(  h_n) \leq  \Delta_{x_0,x_1}(\what h_n^{\rm L})+\Delta_{x_0,x_1}(\what \eps_n), \quad 0\leq x_0<x_1\leq 1.
\end{equation}
\arxivc{
Indeed,   with $n_0=\floor{nx_0}$ and $n_1=\floor{nx_1}$ we have
\begin{multline*}
   \Delta_{x_0,x_1}(  h_n) =\sum_{j=n_0+1}^{n_1} \pp{2\tau_j-1}\leq \sum_{j=n_0+1}^{n_1} \pp{2\what \tau_j-\frac{2}{1+C_n}}+ \sum_{j=n_0}^{n_1} \pp{\frac{2}{1+C_n}-1}
 \\  =\Delta_{x_0,x_1}(\what h_n^{\rm L}) + (n_1-n_0)\frac{1-C_n}{1+C_n}
 = \Delta_{x_0,x_1}(\what h_n^{\rm L})+\Delta_{x_0,x_1}(\floor{n x}) \frac{1-C_n}{1+C_n}.
\end{multline*}
Similarly,
\begin{multline*}
  \Delta_{x_0,x_1}(  h_n) \geq  \sum_{j=n_0+1}^{n_1} \pp{2\wt \tau_j-\frac{2A_n}{1+A_n}}+ (n_1-n_0)\frac{A_n-1}{1+A_n}
\\
 =  \Delta_{x_0,x_1}(\wt h_n^{\rm H})+\Delta_{x_0,x_1}(\floor{n x}) \frac{A_n-1}{1+A_n}.
\end{multline*}

}
Since  $\Delta_{x_0,x_1}(f) =-\Delta_{x_1,x_0}(f)$, \eqref{h-sandwitch} implies that for any $x_0,x_1\in[0,1]$ we have
\begin{equation}\label{Delta-bd}
  |\Delta_{x_0,x_1}(  h_n)|\leq |\Delta_{x_0,x_1}(\wt h_n^{\rm H})|+|\Delta_{x_0,x_1}(\what h_n^{\rm L})|+|\Delta_{x_0,x_1}(\what \eps_n)|+|\Delta_{x_0,x_1}(\wt \eps_n)|.
\end{equation}
From the proof of tightness in the Donsker's theorem we know that
 for any $\eps,\eta>0$ there exists $\delta\in(0,1)$ and $n_0$ such that
\[\proba\pp{\sup_{|x_1-x_0|<\delta}\left\{|\Delta_{x_0,x_1}(\wt h_n^{\rm H})|+|\Delta_{x_0,x_1}(\what h_n^{\rm L})|\right\}\geq \eps \sqrt{n}}\leq \eta, \quad n\geq n_0.
\]
Furthermore, uniformly in $x\in[0,1]$, we have
\begin{equation}\label{eps-lim1}
   \limn\frac{1}{\sqrt{n}} \what \eps_n(x)= \limn\frac{\floor{n x}}{n} \frac{\sqrt{n}(1-C_n)}{1+C_n}= x \frac{\C}{2}
\end{equation}
and
\begin{equation}\label{eps-lim2}
   \limn\frac{1}{\sqrt{n}} \wt \eps_n(x)= \limn\frac{\floor{n x}}{n} \frac{\sqrt{n}(A_n-1)}{1+A_n}=  - x \frac{\A}{2},
\end{equation}
so $\sup_{|x_0-x_1|<\delta}\left(|\Delta_{x_1,x_1}(\hat\eps_n)|+|\Delta_{x_1,x_1}(\tilde\eps_n)|\right)\le \delta(|\A|+|\C|)\sqrt{n}$ for large $n$.
Therefore, \eqref{Delta-bd} implies that
 for any $\eps,\eta>0$ there exists $\delta\in(0,1)$ and $n_0$ such that
\[\proba\pp{\sup_{|x_1-x_0|<\delta}|\Delta_{x_0,x_1}(h_n)|\geq \eps \sqrt{n}}\leq \eta, \quad n\geq n_0.
\]
Since $h_n(0)=0$, this shows, see \cite[Theorem 15.5]{billingsley1968}, that the sequence of processes $\frac{1}{\sqrt{n}}\ccbb{h_n(x)}_{x\in[0,1]}$ is  tight in $D[0,1]$ and the laws of  subsequential weak limits are supported on $C[0,1]$.
\end{proof}

\section{Proof of Theorem \ref{thm:1} for $\A+\C=0$}  \label{sec:a+c=0}
The case of  Theorem \ref{thm:1} for $\A+\C=0$ is a quick application of coupling from Proposition \ref{thm-coupling} combined with tightness from Proposition \ref{Prop-tightness}.
\begin{proof}[Proof of Theorem \ref{thm:1} for $\A+\C=0$]
  Recalling \eqref{eq:hhh}, \eqref{eq:hh} and \eqref{the:eps},  consider
\begin{equation*}\label{eq:hhh+}
\wt h_n(x) := \summ j1{\floor{nx}}\pp{2\wt \tau_j-1}=  \wt h_n^{\rm L}(x)+\wt \eps_n(x) , \quad x\in[0,1],
\end{equation*}
and
\begin{equation*}\label{eq:hh+}
\what h_n(x) := \summ j1{\floor{nx}}\pp{2\what \tau_j-1} = \what h_n^{\rm H}(x)+\what \eps_n(x)
, \quad x\in[0,1].
\end{equation*}
Since $\A+\C=0$, from Donsker's theorem and \eqref{eps-lim1}, \eqref{eps-lim2} we see that both processes have the same limit
  \begin{equation}
    \label{same-limit}
    \frac{1}{\sqrt{n}} \ccbb{\what  h_n(x)}_{x\in[0,1]}\fddto \ccbb{\BB_x -\frac{\A}{2}x}_{x\in[0,1]},  \quad
    \frac{1}{\sqrt{n}} \ccbb{\wt  h_n(x)}_{x\in[0,1]}\fddto \ccbb{\BB_x -\frac{\A}{2}x}_{x\in[0,1]}.
  \end{equation}
  (In fact, convergence is in $D[0,1]$.)
From  Proposition \ref{thm-coupling} it is clear that for $0\leq x \leq 1$ we have
\begin{equation}
  \label{h-sandwitch+}
  \wt h_n(x) \leq   h_n(x) \leq  \what h_n(x).
\end{equation}
In view of \eqref{same-limit}, we see that
\[
\frac{1}{\sqrt{n}} \ccbb{ h_n(x)}_{x\in[0,1]}\fddto \ccbb{\BB_x -\frac{\A}{2}x}_{x\in[0,1]}
\]
and since $\A,\C$ are finite,  by Proposition \ref{Prop-tightness} convergence is in $D[0,1]$.
\arxivc{
One way to prove convergence of finite dimensional distributions is to deduce  from \eqref{h-sandwitch+} inequalities  for the joint cumulative distribution functions. For fixed $\vv x\in[0,1]^d$ and $\vv y\in\RR^d$, define
$$F\topp n _{\vv x}(\vv y) :=\proba\pp{\bigcap_{j=1}^d \ccbb{ \frac{1}{\sqrt{n}}h_n(x_j)\leq y_j}},$$
and similarly, define
$$
\wt F\topp n _{\vv x}(\vv y) :=\proba\pp{\bigcap_{j=1}^d \ccbb{ \frac{1}{\sqrt{n}}\wt h_n(x_j)\leq y_j}}, \quad
\what F\topp n _{\vv x}(\vv y) :=\proba\pp{\bigcap_{j=1}^d \ccbb{ \frac{1}{\sqrt{n}}\what h_n(x_j)\leq y_j}}.$$
Then \eqref{h-sandwitch+} gives
$$\what F\topp n_{\vv x}(\vv y)\leq F\topp n_{\vv x}(\vv y)\leq \wt F\topp n_{\vv x}(\vv y). $$
Since by \eqref{same-limit}, the CDFs  $\wt F\topp n _{\vv x}(\vv y)$ and $\what  F\topp n _{\vv x}(\vv y)$  converge to the same limit, we get convergence of
 $F\topp n _{\vv x}(\vv y)$.
}
\end{proof}

\section{Proof of Theorem \ref{thm:1} for $\A+\C>0$}  \label{sec:a+c>0}
In view of Proposition \ref{C:swap}, it is enough to prove Theorem \ref{thm:1}, under an additional assumption $A_n\ge C_n$, which implies $\A\le \C$.
Note   that this excludes the case $\A = \infty, \C\in\R$ but includes the cases $\A=\C=\infty$ and $\A\in\RR, \C=\infty$. For   convenience, we restate Assumption \ref{assump:0}   with these additional constraints on $\A, \C$ stated explicitly.
\begin{assumption}\label{assump:1}
 We assume that  $A_n,C_n\ge 0$, $A_nC_n<1$, $A_n\geq C_n$,  $B_n,D_n\in(-1,0]$ for all $n\geq 1$. Moreover,  we assume that
\[
\limn C_n = 1 \qmand  \limn \sqrt n(1-C_n) = \C\in(0,\infty], \]
and
\[
\limn A_n = 1 \qmand \limn \sqrt n(1-A_n) = \A\in(-\C,\C],
\]
(so  $\A+\C>0$).
We also assume that \eqref{limBnDn} and \eqref{bd:BnDn} hold.

\end{assumption}
 For $d\in\N$, and $\vv x=(x_1,\dots,x_d)$
 with
$x_0:=0< x_1<\cdots<x_d=1$,
 we introduce the Laplace transform $\varphi_{\vv x,n}$ by the formula
\begin{align}\label{def:Lap}
\varphi_{\vv x,n}(\vv c) & := \aa{\exp\pp{-\summ k1d c_kh_n(x_k)}}_n, \quad \vv c=(c_1,\dots,c_d) \in \R^d.
\end{align}

Below we shall compute the limit of $\varphi_{\vv x,n}(\vv c/\sqrt n)$ as $n\to\infty$. By \cite[Appendix]{bryc19limit}, this shall determine the convergence  of the finite-dimensional distributions, even though  the coefficients $\vv c=(c_1,\dots,c_d)$ will be  restricted to an open subset  of $\RR^d$ determined by inequalities \eqref{eq:a+c} .

 The limit is expressed in terms of the process,    called  the $1/2$-stable Biane process in \cite{bryc16local},
which can also be defined as a square of  the  radial  part  of  a  3-dimensional  Cauchy  process \cite[Corollary 1]{KyprianouOConnell2021}.
 This is a time-homogeneous Markov process taking values in $(0,\infty)$ with transition probability density function
\begin{equation}\label{eq:Biane p}
\p_{t}(u,v) = \frac{2t\sqrt v}{\pi[t^4+2t^2(u+v)+(u-v)^2]}\inddd{u,v>0},\; t>0.
\end{equation}
 \begin{theorem}\label{thm:2} Under Assumption \ref{assump:1}, for all $\vv c=(c_1,\dots,c_d) \in\RR^d$ such that   %
\begin{equation}\label{eq:a+c}
 c_1,\dots,c_{d-1}>0  \mbox{ and }
-\A<c_d<c_1+\cdots+c_d<\C,
\end{equation}
we have
\begin{equation}
  \label{konkluzja}
  \limn\varphi_{\vv x,n}\pp{\frac{\vv c}{\sqrt n}}  %
= \esp \exp\pp{-\summ k1d \frac{c_k}{\sqrt 2}\BB_{x_k}}  \cdot \Psi\topp{\A,\C}_\vvx\pp{\vvc},
\end{equation}
with $\Psi_\vvx\topp{\A,\C}$ which has three different  forms as follows:
\begin{equation}
\label{PsiAC}
\Psi_\vvx\topp{\A,\C}(\vvc)  =
\begin{cases}
\displaystyle\frac{\sqrt{2}}{\pi \mathfrak C_{\A,\C}} \int_{\R_+^d}\frac{\sqrt {u_1}}{((\C-s_1)^2+u_1)((\A+c_d)^2+u_d)} L_{\vvx,\vvc}(\vvu)\,  d\vv u   & \A,\C\in\RR,
\\
\\
\displaystyle\frac1{2\sqrt{2}\,\pi\mathfrak C_{\A,\infty}}\int_{\R_+^d}\frac{\sqrt{u_1}}{(\A+c_d)^2+u_d}  L_{\vvx,\vvc}(\vvu) \, d\vv u  & \A\in\RR, \C=\infty,\\
 \\
\displaystyle \frac1{4\sqrt \pi} \int_{\R_+^d} \sqrt{u_1} L_{\vvx,\vvc} (\vvu) \, d\vv u  & \A=\infty,\C=\infty,
\end{cases}
\end{equation}
where %
$s_1=c_1+\dots+c_d$,  constant
$\mathfrak C_{\A,\C}$ is defined in \eqref{eq:C_ac},
and
\begin{equation}
    \label{eq:Gx1}
   L_{\vvx,\vvc}(\vvu) = \exp\pp{-\frac14\summ k1d(x_k-x_{k-1})u_k}\prodd k1{d-1}\p_{c_k}(u_k,u_{k+1}).
\end{equation}

\end{theorem}
The proof of Theorem \ref{thm:2} consists of several steps and is presented in Section \ref{sec:Proof:T2}.

The second %
  step in   the proof of Theorem \ref{thm:1}   for $\A+\C>0$   is to  express  $\Psi_\vvx\topp{\A,\C}$  as the Laplace transform of a stochastic process. We will show the following.
\begin{proposition}\label{prop:duality summary}
Fix $\A\leq \C$. For  $d\in\N$, $\vv x=(x_1,\dots,x_d)$
 with
$x_0:=0< x_1<\cdots<x_d=1$, and $\vv c=(c_1,\dots,c_d) \in \R^d$ such that \eqref{eq:a+c}  holds, %
we have
\[\Psi\topp{\A,\C}_\vvx(\vvc)  = \esp\exp\pp{-\frac1{\sqrt 2}\summ k1d c_k\eta\topp{\A,\C}_{x_k}},\]
where $\eta\topp{\A,\C}$ is defined in Section \ref{sec:eta}.
\end{proposition}
We shall deduce Proposition \ref{prop:duality summary} from a more general Proposition \ref{prop:duality} in Section \ref{sec:duality}.

Assuming these two results, we can now prove   Theorem \ref{thm:1} for $\A+\C>0$.

\begin{proof}[Proof of Theorem \ref{thm:1} for $\A+\C>0$]
By
Proposition \ref{C:swap} we can replace
Assumption \ref{assump:0} by Assumption \ref{assump:1} and apply Theorem \ref{thm:2}.

By Theorem \ref{thm:2}, the Laplace transform
\eqref{def:Lap} converges to the product of two Laplace transforms.
In view of Proposition \ref{prop:duality summary}, we recognize the Laplace transform of the desired limit.
To conclude the proof, we invoke the fact  that convergence of Laplace transforms of probability measures
to a Laplace transform of a probability measure on an open set of arguments $\vv c$ implies weak convergence of measures, \cite[Theorem A.1]{bryc19limit}. This proves convergence of finite dimensional distributions.

When $\A,\C$ are finite, convergence in $D[0,1]$ follows from Proposition \ref{Prop-tightness}.
\end{proof}
It remains to prove Theorem \ref{thm:2} and  Proposition \ref{prop:duality summary}.

\subsection{Proof of Theorem \ref{thm:2}}\label{sec:Proof:T2}
We start by expressing the Laplace transform
\begin{align*}
\varphi_{\vv x,n}(\vv c)
& = \aa{\exp\pp{-\summ k1d\sum_{j=\floor{nx_{k-1}}+1}^{\floor{nx_k}}\pp{2\tau_j-1}(c_k+\cdots+c_d)}}_n
\end{align*}
in terms of   the Askey--Wilson process   $(Y\topp n_t)_{t\geq 0}$  with parameters $A_n,B_n,C_n,D_n,q$, which has marginal  laws \eqref{eq:AW marginal} and   transition probabilities  \eqref{eq:AW transition}. %
Note that the range  of the legitimate time index of the Askey--Wilson process $Y\topp n$ depends on the parameters, and under Assumption \ref{assump:1}, the domain $[0,\infty)$ follows from \cite[
Eq.~(1.21)]{bryc10askey}.

Write
 \begin{equation*}\label{eq:s_k}
 s_k: = c_k+\cdots+c_d,   \quad  k=1,\dots,d, \quad \mand \quad n_k :=\sfloor{n{x_k}}, \quad  k=0,1,\dots, d.
 \end{equation*}
(Recall that  $0=x_0<x_1<x_2<\dots<x_d=1$.)

By \eqref{eq:BW} we have,
\begin{align*}
\varphi_{\vv x,n}(\vv c) %
& =  \exp\pp{\summ k1d  s_k(n_k-n_{k-1})}\aa{{\prodd k1{d}\prod_{j=n_{k-1}+1}^{n_k}(e^{-2 s_k})^{\tau_j}}}_n\\
& = \exp\pp{\summ k1d  s_k(n_k-n_{k-1})}\frac{\esp\bb{\prodd k1{d}(1+e^{-2 s_k}+2e^{- s_k}Y_{e^{-2 s_k}}\topp n )^{n_k-n_{k-1}}}}{2^n\esp(1+Y_1\topp n)^n}\\
& = \frac{\esp\bb{\prodd k1{d} \spp{\cosh(s_k)+Y_{e^{-2s_k}}\topp n}^{n_k-n_{k-1}}}}{\esp(1+Y_1\topp n)^n}.
\end{align*}
Then,
\begin{equation}\label{eq:Laplace}
\varphi_{\vv x,n}\pp{\frac{\vv c}{\sqrt n}} = \frac{\esp\bb{\prodd k1{d} \pp{\cosh(s_k/\sqrt n)+Y\topp n_{e^{-2s_k/\sqrt n}}}^{n_k-n_{k-1}}}}{\esp(1+Y_1\topp n)^n}.
\end{equation}
We shall consider separately the numerator and the denominator.

\subsubsection{Asymptotics of the numerator in \protect{\eqref{eq:Laplace}}}
To analyze the numerator we first denote
\begin{equation*}
    t_{k,n}=e^{-2 s_k/\sqrt{n}}, k=1,\dots, d.
\end{equation*}
It follows from  \eqref{eq:a+c} that for $n$ large enough,
\begin{equation}
\label{eq:A<1}
  A_n\sqrt{t_{k,n}}<1,\; |B_n\sqrt{t_{k,n}}| <1,\;C_n/\sqrt{t_{k,n}}<1,
\end{equation}
 and hence they  do not create atoms for the  marginal law \eqref{eq:AW marginal} at time $t=t_{k,n}$.

However $D_n$ can create atoms,  or more precisely $-D_n/\sqrt{t_{k,n}}>1$ is possible if $D=-1$.

In this case, the atom of the marginal law of %
$Y\topp n_{t_{k,n}}$ is created by $D_n$ with value
\begin{equation*}
\frac12\pp{D_n/\sqrt{t_{k,n}}+\frac1{D_n/\sqrt{t_{k,n}}}} =\frac{(1+D_n/\sqrt{t_{k,n}})^2}{2D_n/\sqrt{t_{k,n}}} -1 <-1.
\end{equation*}
 Consequently, since $D_n/\sqrt{t_{k,n}}\to -1$ when $D=-1$, for a fixed $\delta\in(0,1)$ and large enough $n$, we have  %
 \begin{equation}\label{eq:Y-range}
     Y\topp n_{t_{1,n}} \in (-1-\delta, 1].
 \end{equation}
 \arxivc{
 Indeed, $D_n\to D=-1$ and $t_{k,n}\to 1$}
In view of the first two inequalities in \eqref{eq:A<1}, we also do not have atoms
  in the  transition probabilities  \eqref{eq:AW transition} from $s=t_{k,n}$ to $t=t_{k+1,n}$, when starting at  $y\in [-1,1]$.

Using \eqref{eq:Y-range}, for large enough $n$ we write
\begin{equation}\label{eq:I1+I2}
\esp  \bb{\prodd k1{d} \pp{\cosh(s_k/\sqrt n)+Y\topp n_{e^{-2s_k/\sqrt n}}}^{n_k-n_{k-1}}}  = I_{n,1} + I_{n,2}
\end{equation}
with %
\begin{align*}I_{n,1} & :=
\esp\bb{\ind_{Y\topp n_{t_{1,n}}\in(-1-\delta,0)}\prodd k1{d} \pp{\cosh(s_k/\sqrt n)+Y\topp n_{e^{-2s_k/\sqrt n}}}^{n_k-n_{k-1}}},
\\
I_{n,2}& :=\esp\bb{\ind_{Y\topp n_{t_{1,n}}\in[0,1]}\prodd k1{d} \pp{\cosh(s_k/\sqrt n)+Y\topp n_{e^{-2s_k/\sqrt n}}}^{n_k-n_{k-1}}}.
\end{align*}

We first analyze $I_{n,2}$,
which is a multivariate integral with $Y\topp n_{t_{1,n}}\in[0,1]$.
We introduce the process
\begin{equation*}\label{eq:wtY_n}
\wt Y\topp n_t:=2n\pp{1-Y_{e^{2t/\sqrt n}}\topp n},\qquad  t\in(-\infty,\infty).
\end{equation*}
 Since $Y\topp n_{t_{1,n}}\in[0,1]$, therefore $Y\topp n_{t_{k,n}}\in[-1,1]$, $k=2,\dots,d$. Hence,
we have $\wt Y\topp n_{t_{1,n}}\in[0,2n]$ and $\wt Y\topp n_{t_{k,n}}\in[0,4n]$, $k=2,\dots,d$.

 We shall   use the notations $\wt \pi_t\topp n$ ($\pi_t\topp n$ respectively),
for the absolutely continuous parts of the marginal   laws   of $\wt Y\topp n$ ($Y\topp n$ respectively)  which are supported on  $[0,4n]$ ($[-1,1]$ respectively).

Recalling \eqref{eq:f}  we write
\[
\wt \pi_t\topp n (y)=\frac{1}{2n}\pi_{e^{2t/\sqrt{n}}}\topp n\left(1-\frac{y}{2n}\right) \mbox{ with }
\pi_t\topp n(y)= f(y; A_n \sqrt{t},B_n\sqrt{t},C_n/\sqrt{t},D_n/\sqrt{t},q).
\]

Using this density (which may be sub-probabilistic) and \begin{align}
G_{\vvx,\vvc,n}(\vv u)& := 2^{-n}\prodd k1{d} \pp{\ind_{[0,4n]}(u_k)\pp{\cosh\pp{\frac{s_k}{\sqrt n}}+1-\frac{u_k}{2n}}^{n_k-n_{k-1}} }\label{**G**}
\\
& = \prodd k1{d}\pp{\ind_{[0,4n]}(u_k)\pp{1+\sinh^2\pp{\frac{s_k}{2\sqrt n}} - \frac {u_k}{4n}}^{n_k-n_{k-1}}},\label{eq:Gxcn}
\end{align}

we   can   write
\begin{equation}\label{eq:I_2}
    \frac{I_{n,2}}{2^n}=\int \ind_{[0,2n]}(u)
     \esp\pp{G_{\vvx,\vvc,n}(u,\wt Y_{-s_2}\topp n,\dots,\wt Y_{-s_d}\topp n)\mmid \wt Y_{-s_1}\topp n = u} \wt \pi_{-s_1}\topp n (u)\, du.
\end{equation}
\arxivc{Indeed, using \eqref{**G**} we see that
\eqref{eq:I1+I2} says that
\begin{align*}
    I_{n,2} & =\esp\bb{\ind_{Y\topp n_{t_{1,n}}\in[0,1]}\prodd k1{d} \pp{\cosh(s_k/\sqrt n)+Y\topp n_{e^{-2s_k/\sqrt n}}}^{n_k-n_{k-1}}}
    \\& = 2^n  \esp\bb{\ind_{[0,2n]}(\wt Y_{-s_1}\topp n) G_{\vvx,\vvc,n}(\wt Y_{-s_1}\topp n,\wt Y_{-s_2}\topp n,\dots,\wt Y_{-s_d}\topp n) }.
\end{align*}
}
In the next two lemmas we   address the asymptotics of the density $\wt\pi_{t}\topp n$ and of the transition probability density for the process $(\wt Y\topp n_t)$.
\begin{lemma}\label{lem:1}
Under Assumption \ref{assump:1}, for every $-\C<t<\A$ we have the following.

(i) Pointwise convergence: for every $u\geq 0$, %
\begin{equation}\label{eq:pi}
\limn\frac{\wt \pi\topp n_t(u)}{\mathsf \Pi_n}=  \begin{cases}
\displaystyle
\frac{\sqrt u}{((\A-t)^2+u)((\C+t)^2+u)}, & \mbox{ if } \A,\C\in\R,\\\\
\displaystyle \frac{\sqrt u}{(\A-t)^2+u}, & \mbox{ if } \A \in\R,\C = \infty,\\\\
\displaystyle \sqrt u, & \mbox{ if } \A =\C= \infty, \end{cases}
\end{equation}
with
\begin{equation}\label{eq:Pi_n}
\mathsf \Pi_n = \frac{1}{\pi}\frac{1-B_nD_n}{1-A_nB_nC_nD_n}\cdot \begin{cases}
\displaystyle \A+\C , & \mbox{ if } \A,\C\in\R,\\\\
\displaystyle\frac1{\sqrt n(1-C_n)}, & \mbox{ if } \A\in\R, \C = \infty, \\\\
\displaystyle \frac{1-A_nC_n}{n^{3/2}(1-A_n)^2(1-C_n)^2}, & \mbox{ if } \A=\C = \infty. %
\end{cases}
\end{equation}
(ii)  Uniform bound:
There are constants $N$,    $K>0$ (possibly depending on $t$)   such that for all  $n>N$ and all $u\geq 0$,
\begin{equation}
    \label{eq:pi upper}
    \ind_{u\in[0,2 n]}\frac{\wt\pi_t\topp n(u)}{\mathsf \Pi_n}\leq  K  \sqrt{u}.
\end{equation}
\end{lemma}
\begin{remark}The exact asymptotic of $\mathsf \Pi_n$ is not needed for the proof of our main theorem, because of a cancellation later in the calculations.
However we need the following consequence of Assumption \ref{assump:1}
\begin{equation}
    \label{eq:Pi-bd}
    \limsup_{n\to\infty} \frac1n \log (1/\mathsf \Pi_n) \leq 0.
\end{equation}

Note that unless $BD=1$, we have $\frac{1-B_nD_n}{1-A_nB_nC_nD_n}\sim 1$,
which further simplifies   expression \eqref{eq:Pi_n}.
\end{remark}
\arxivc{\begin{proof}[Proof of \eqref{eq:Pi-bd}]
We use \eqref{bd:BnDn} to handle the pre-factor $$
\frac{1-A_nB_nC_nD_n}{1-B_nD_n}\leq \frac{1}{1-B_nD_n}.
$$
Similarly, we use the trivial bound  $\sqrt{n}(1-C_n)\leq \sqrt{n}$ for the middle expression with $\C=\infty$ in \eqref{eq:Pi_n}. For the last case,  since $A_n\geq C_n>0$, we have
$$
\frac{(1-A_n)^2(1-C_n)^2}{(1-A_nC_n)}\leq \frac{(1-A_n)^2(1-C_n)^2}{(1-A_n^2)} = \frac{(1-A_n)(1-C_n)^2}{(1+A_n)}\leq \frac{1}{(1+A_n)}\sim 1/2.$$
\end{proof}
}
\begin{proof}[Proof of Lemma \ref{lem:1}]
Consider %
\begin{equation*}\label{x_n and t_n}
  x_n = 1-\frac u{2n},\quad t_n = e^{2t/\sqrt n}.
\end{equation*}
 (Strictly speaking, $t_n$ is a function of $t$.)
Since $(\alpha;q)_\infty=(1-\alpha)(\alpha q;q)_\infty$, from \cite[Eq.~(4.4)]{bryc19limit},
\begin{equation}
\label{eq:pi_tn}
\wt \pi\topp n_t(u)  = \frac1{2n}\pi_{t_n}(x_n; A_n,B_n,C_n,D_n,q)
=\frac1{\pi n} J_n(t,u)  R_n(t,u)
\end{equation}
with
\begin{align*}J_n(t,u) &:= J\pp{1-\frac{u}{2n};A_n\sqrt{t_n},B_n\sqrt{t_n},C_n/\sqrt{t_n},D_n/\sqrt{t_n}},\\
R_n(t,u) & :=R\pp{1-\frac{u}{2n};A_n\sqrt{t_n},B_n\sqrt{t_n},C_n/\sqrt{t_n},D_n/\sqrt{t_n},q},
\end{align*}
where $J$ and $R$ are given by \eqref{eq:J} and \eqref{eq:R}.
Applying  \eqref{1*1} to each factor in
\[
    \label{eq:R_n(t,u)}
R_n(t,u)
=
\frac{\abs{\qp{q e^{2i\theta_{x_n}}}}^2}{\qp{qA_nB_nC_nD_n}}
\cdot \frac{\qp{qA_nB_nt_n, qA_nC_n, qA_nD_n, qB_nC_n, q B_nD_n, q C_nD_n/t_n}}{\abs{\qp{qB_n\sqrt{t_n}e^{i\theta_{x_n}},q D_ne^{i\theta_{x_n}}/\sqrt{t_n},q A_n\sqrt{t_n}e^{i\theta_{x_n}},q C_ne^{i\theta_{x_n}}/\sqrt{t_n}}}^2},
\]
we see that
\begin{equation} \label{eq:Rn}
R_n(t,u)\to 1.
\end{equation}

The explicit formula for $J_n(t,u)$ is
\begin{multline*}\label{eq:J_n-expl}
J_n(t,u)=\frac{ \sqrt{1-x_n^2}}{\abs{(1-A_n\sqrt{t_n}e^{i\theta_{x_n}})(1-C_ne^{i\theta_{x_n}}/\sqrt{t_n})}^2}
\\
\times \frac{(1-A_nB_nt_n)(1- A_nC_n)(1-A_nD_n)(1-B_nC_n)(1- B_nD_n)(1- C_nD_n/t_n)}{(1-A_nB_nC_nD_n)|(1-B_n\sqrt{t_n}e^{i\theta_{x_n}})(1-D_ne^{i\theta_{x_n}}/\sqrt {t_n})|^2}.
\end{multline*}
Since
\[
   \frac{(1-A_nB_nt_n)(1-B_nC_n)(1-A_nD_n)(1- C_nD_n/t_n)}{\left|(1-B_n\sqrt{t_n}e^{i\theta_{x_n}})(1-D_ne^{i\theta_{x_n}}/\sqrt{t_n})\right|^2}
   \sim \frac{(1-B)(1-B)(1-D)(1- D)}{(1-B)^2(1-D)^2} = 1,
\]
these factors  do not contribute to the asymptotics.
 Noting that
$\sqrt{1-x_n^2}\sim \sqrt{u/n}$, we get
\begin{equation*}
 J_n(t,u)\sim \frac{\sqrt{u}}{\sqrt{n}}\cdot  \frac{1-B_nD_n}{1-A_nB_nC_nD_n} \cdot \frac{1-A_nC_n}{
\abs{(1-A_n\sqrt{t_n}e^{i\theta_{x_n}})(1-C_ne^{i\theta_{x_n}}/\sqrt{t_n})}^2}.
\end{equation*}

We observe that
\begin{equation}\label{1-AC}
1-A_nC_n \sim \begin{cases}
\displaystyle\frac{\A+\C}{\sqrt n} , & \mbox{ if } \A,\C\in\R,\\\\
1-C_n,& \mbox{ if } \A\in\R, \C = \infty,
\end{cases}
\end{equation}
(keeping the factor $1-A_nC_n$ unchanged when  $\A=\C = \infty$)
and
\begin{equation}\label{eq:An1}
\abs{1-A_n\sqrt{t_n}e^{i\theta_{x_n}}}^2
 =  (1-A_n\sqrt{t_n})^2 + 2A_n\sqrt{t_n}(1-{x_n})= (1-A_n\sqrt{t_n})^2+A_n\sqrt{t_n} u/n.%
\end{equation}
Thus, %
\arxivc{noting that if $\A=\infty$ then
$$
\frac{1-A_n\sqrt{t_n}}{1-A_n}=1+A_n \frac{\sqrt{n}(1-\sqrt{t_n})}{\sqrt{n}(1-A_n)}\to 1,
$$
we get
}

\[
\abs{1-A_n\sqrt{t_n}e^{i\theta_{x_n}}}^2 \sim
\begin{cases}
\displaystyle \frac{(\A-t)^2+u}n, &\mbox{ if } \A\in\R,\\\\
(1-A_n)^2, & \mbox{ if } \A = \infty.
\end{cases}
\]
Similarly
\[
\abs{1-C_n/\sqrt{t_n}e^{i\theta_{x_n}}}^2 \sim
\begin{cases}
\displaystyle \frac{(\C+t)^2+u}n, &\mbox{ if } \C\in\R,\\\\
(1-C_n)^2, & \mbox{ if } \C = \infty.
\end{cases}
\]
Thus
\begin{equation}
    J_n(t,u)
  \sim  \frac{1-B_nD_n}{1-A_nB_nC_nD_n} \cdot \begin{cases}
\displaystyle
n^{3/2}\cdot  \frac{\sqrt u}{((\A-t)^2+u)((\C+t)^2+u)}, & \mbox{ if } \A,\C\in\R,\\\\
\displaystyle \sqrt n\cdot  \frac1{ (1-C_n) }\cdot\frac{\sqrt u}{(\A-t)^2+u}, & \mbox{ if } \A \in\R,\C = \infty,\\\\
\displaystyle \frac1{\sqrt n}\cdot  \frac{ 1-A_nC_n }{(1-A_n)^2(1-C_n)^2}\cdot\sqrt u, & \mbox{ if } \A =\C= \infty. \end{cases}
\label{eq:Jn}
\end{equation}

Combining \eqref{eq:pi_tn}, \eqref{eq:Rn} and  \eqref{eq:Jn}    yields the desired asymptotic equivalence \eqref{eq:pi} for $\wt\pi\topp n_t$.

Next we prove part (ii).
From \eqref{2*2} with $\alpha_n$ replaced by $q\alpha_n$, and
elementary calculations, like
\[
\sup_{u\in[0,4n]}|1-q^k A_n\sqrt{t_n} e^{i\theta_{x_n}}|^2=- (1-q^k A_n\sqrt{t_n})^2,
\]
compare \eqref{eq:An1}, we  see that the sequence $\sup_{u\in[0,4n]}R_n(t,u)$ converges as $n\to \infty$.
 \arxivc{
 A somewhat more subtle point here is the lower  bound for the factors in the denominator  like
 $$\inf_{u\in[0,4n]}\abs{1-q^k B_n\sqrt{t_n} e^{i\theta_{x_n}}}^2=
 \inf_{u\in[0,4n]}(1-q^k B_n\sqrt{t_n})^2+q^kB_n\sqrt{t_n}\,\frac{u}{n} =
 (1+q^k B_n\sqrt{t_n})^2,$$
 where, since $B_n\leq 0$,  the minimum is attained at $u=4n$. For large enough $n$, this is strictly positive, as $q^k\leq q$ and $B_n>-1$.
 }
 Thus  there exists $N>0$ and a constant $K>0$  (here and in the remaining part of this proof $K$ stands for a generic constant which may change from line to line, may depend on $t$, but is free of $n$ and $u$)   such that
for   all $u\in[0,4n]$ and all $n>N$ we have
 \begin{equation}\label{bd:R+}
  0\leq R_n(t,u)\leq K.
\end{equation}

The bound for $J_n(t,u)$ is slightly more delicate and introduces an additional constraint on $u$. Notice that for all $u\in[0, 2 n]$   (so $x_n\ge 0$),
\begin{equation*}
  \label{lbd:B}
\abs{1-B_n\sqrt{t_n}e^{i\theta_{x_n}}}^2= 1-2 B_n \sqrt{t_n}x_n+B_n^2t_n\geq    1.
\end{equation*}
Similarly,
\begin{equation*}
  \label{lbd:D}|1-D_n e^{i\theta_{x_n}}/\sqrt{t_n}|^2\geq 1,\quad \mbox{for }\quad u\in[0,2n].
\end{equation*}

Note also that $\sqrt{1-x_n^2}\leq \sqrt{u/n}$. Thus, recalling \eqref{eq:An1} and its analogue for $C_n$ we conclude that for $N$ large enough and $n>N$ there exists  $K$ such that
\begin{equation}\label{eq:jntu}
0\leq J_n(t,u)\le K\frac{\sqrt{u}(1-B_nD_n)(1-A_nC_n)}{\sqrt{n}(1-A_nB_nC_nD_n)(1-A_n\sqrt{t_n})^2(1-C_n/\sqrt{t_n})^2}, \quad u\in[0,2n].
\end{equation}

Moreover, due to \eqref{1-AC}, \eqref{eq:An1} and  its analogue for $C_n$ we can increase $N$  to ensure that   for all $n>N$ we have
\begin{equation}\label{eq:ac}
1-A_nC_n\le \left\{\begin{array}{ll} \displaystyle2\frac{\A+\C}{\sqrt{n}}, & \mbox{ if }\;\A,\C<\infty, \\\\
2(1-C_n), & \mbox{ if }\; \A<\infty, \,\C=\infty,\end{array}\right.
\end{equation}
\begin{equation}
  \label{lbd:A+}
  \abs{1-A_n\sqrt{t_n}}^2 \geq  \left\{\begin{array}{ll} \displaystyle \frac{(\A-t)^2}{2n}, & \mbox{ if }\;\A<\infty, \\\\
  \displaystyle\frac{1}{2}(1-A_n)^2, & \mbox{ if }\,\A=\infty,
  \end{array}\right.
\end{equation}
and
\begin{equation}
  \label{lbd:C+}
  \abs{1-C_n/\sqrt{t_n}}^2 \geq  \left\{\begin{array}{ll} \displaystyle \frac{(\C+t)^2}{2n}, & \mbox{ if }\;\C<\infty, \\\\
\displaystyle  \frac{1}{2}(1-C_n)^2, & \mbox{ if }\,\C=\infty.
  \end{array}\right.
\end{equation}

Inserting \eqref{eq:ac}, \eqref{lbd:A+} and \eqref{lbd:C+}  into \eqref{eq:jntu} we thus obtain
\begin{equation}\label{bd:J+}
\ind_{u\in[0,2n]} J_n(t,u)\leq   K\sqrt{u}\frac{1-B_nD_n}{1-A_n B_n C_n D_n}\left\{\begin{array}{ll} \displaystyle\frac{n}{(\A-t)^2(\C+t)^2}, & \mbox{ if }\A,\C<\infty,\\\\
 \displaystyle\frac{\sqrt{n}}{(\A-t)^2(1-C_n)}, & \mbox{ if }\A<\infty,\;\C=\infty,\\\\
 \displaystyle\frac{1-A_n C_n}{\sqrt{n}\,(1-A_n)^2(1-C_n)^2}, & \mbox{ if } \A=\C=\infty.\end{array}\right.
\end{equation}

Recalling \eqref{eq:Pi_n} we see that \eqref{bd:J+} yields $J_{n}(t,u)\le K\,\mathsf \Pi_n\, n\sqrt{u}$ for all $u\in[0,2n]$, which, in view of \eqref{bd:R+} and \eqref{eq:pi_tn},  proves \eqref{eq:pi upper}.
\end{proof}

Next we address the asymptotics of the transition densities of $\wt Y\topp n$. %
 We shall   use the notations $ \wt p\topp n_{s,t}$ and $ p\topp n_{s,t}$,
for the density of the
 transition probabilities   of $\wt Y\topp n$ and $Y\topp n$, respectively. These densities are supported on  the intervals $[0,4n]$ and $[-1,1]$ respectively.

\begin{lemma}\label{lem:2}
Under Assumption \ref{assump:1} we have, for all $s,t$ such that $s<t<\A$, %
\begin{equation}
\lim_{n\to\infty}\wt p_{s,t}\topp n(u,v)
=
\begin{cases}\displaystyle\p_{t-s}(u,v)\frac{(\A-s)^2+u}{(\A-t)^2+v}, &  \mbox{ if } \A<\infty,\\ \\
\p_{t-s}(u,v), &  \mbox{ if } \A=\infty,
\end{cases}\label{eq:pst}
\end{equation}
and $\p_t(x,y)$ is the transition probability density function of the $1/2$-stable Biane process as in \eqref{eq:Biane p}.
 \end{lemma}

\begin{proof}
Consider $x_n = 1-u/(2n), y_n = 1-v/(2n)$.  Recall $s_n<t_n = e^{2t/\sqrt n}$, and hence  $A_n\sqrt {s_n}<A_n\sqrt{t_n}\le 1$  eventually
 by assumption as $s<t<\A$.
 In this case definition \eqref{eq:AW  transition}   gives  the following transition probability density
of the Askey--Wilson process:

\begin{equation}\label{*p*}
   \wt p_{s,t}\topp n(u,v)=  \frac{1}{2n} p_{s_n,t_n}\topp n(x_n,y_n)
=\frac{1}{\pi n}J_n(s,t,u,v)R_n(s,t,u,v),
\end{equation}
where
\begin{align*}
  J_n(s,t,u,v  )&=J\pp{y_n;A\sqrt{t_n},B_n\sqrt t_n, \sqrt{\frac{s_n}{t_n}}\,e^{i\theta_{x_n}},\sqrt{\tfrac{s_n}{t_n}}\,e^{-i\theta_{x_n}}},\\
 R_n(s,t,u,v) &= R\pp{y_n;A\sqrt{t_n},B_n\sqrt t_n, \sqrt{\frac{s_n}{t_n}}\,e^{i\theta_{x_n}},\sqrt{\tfrac{s_n}{t_n}}\,e^{-i\theta_{x_n}},q},
 \end{align*}
with functions  $J$ and $R$ introduced in \eqref{eq:J} and \eqref{eq:R}.

Similarly as in the proof of Lemma \ref{lem:1}, by  \eqref{1*1} we see that
$ R_n(s,t,u,v)\to 1$.

Next we consider
\newcommand{\qpP}[1]{1-#1}
\begin{multline*}
    J_n(s,t,u)=\frac{1-A_nB_nt_n}{1-A_nB_ns_n}\cdot\frac{|1-B_n\sqrt{s_n}e^{i\theta_{x_n}}|^2}{|1-B_n\sqrt{t_n}e^{i\theta_{y_n}}|^2} \\
    \times
    \frac{|1-A_n\sqrt{s_n}e^{i\theta_{x_n}}|^2}
{|\qpP{A_n\sqrt{t_n}e^{i\theta_{y_n}}}|^2} \cdot
\frac{\sqrt{1-y_n^2}(\qpP{s_n/t_n})}{\left|(\qpP{\sqrt{s_n/t_n}e^{i(\theta_{x_n}+\theta_{y_n})}})(\qpP{\sqrt{s_n/t_n}e^{i(-\theta_{x_n}+\theta_{y_n})}})\right|^2}.
\end{multline*}
Note that
\[
\limn\frac{1-A_nB_nt_n}{1-A_nB_ns_n}\cdot\frac{|1-B_n\sqrt{s_n}e^{i\theta_{x_n}}|^2}{|1-B_n\sqrt{t_n}e^{i\theta_{y_n}}|^2}  = 1,
\]
and
by \eqref{eq:An1}, we have
\begin{equation*}\label{eq:An}
\limn\frac{\abs{1-A_n\sqrt{s_n}e^{i\theta_{x_n}}}^2}{
\abs{1-A_n\sqrt{t_n}e^{i\theta_{y_n}}}^2}
 = \frac{(\A-s)^2+u}{(\A-t)^2+v},
\end{equation*}
where  the ratio is understood as 1 if $\A = \infty$.

Moreover, as in \cite[P.~2185]{bryc19limit}, we have
\begin{align*}
  |(\qpP{\sqrt{s_n/t_n}e^{i(\pm \theta_{x_n}+\theta_{y_n})}}) |^2 & =  1+t_n/s_n - 2\sqrt{t_n/s_n}\pp{x_ny_n \mp  \sqrt{1-x_n^2}\sqrt{1-y_n^2}}
 \\& \sim \frac1n\bb{(t-s)^2+(\sqrt u\pm \sqrt v)^2}.
\end{align*}
Since $\sqrt{1-y_n^2}(\qpP{s_n/t_n}) \sim 2(t-s)\sqrt{v}/n$, we get
\begin{equation*}
  J_n(s,t,u)\sim  2n  \frac{(t-s)\sqrt{v}}{((t-s)^2+(\sqrt u+ \sqrt v)^2)((t-s)^2+(\sqrt u- \sqrt v)^2)} \cdot \frac{(\A-s)^2+u}{(\A-t)^2+v}.
\end{equation*}
Recalling \eqref{*p*} and \eqref{eq:Biane p}, the desired limit \eqref{eq:pst} now follows.
\end{proof}
To determine the asymptotic of $I_{n,2}/2^n$ in \eqref{eq:I_2}, we need the  following lemma
about the limiting  properties of  function $G_{\vvx,\vvc,n}$ defined in \eqref{**G**}.
\begin{lemma}\label{lem:3}
If  Assumption \ref{assump:1}
and \eqref{eq:a+c} hold,
then %
\equh
    \label{eq:num_limit}
  \limn \frac1{\mathsf \Pi_n}\int_0^{2n}
\esp\pp{G_{\vvx,\vvc,n}(u,\wt Y_{-s_2}\topp n,\dots,\wt Y_{-s_d}\topp n)\mmid \wt Y_{-s_1}\topp n = u}
\wt \pi_{-s_1}\topp n (u) du
= \exp\pp{\frac14\summ k1{d}s_k^2(x_k-x_{k-1})} \Psi_\vvx\topp{\A,\C}(\vvc),
\eque
where $\Psi_\vvx\topp{\A,\C}(\vvc) $ is defined in \eqref{PsiAC} and \eqref{eq:Gx1}.

\end{lemma}

\begin{proof} %
 Since $\sinh^2 x=x^2+O(x^4)$ as $x\to 0$, we have
\[
\pp{1+\sinh^2\pp{\frac{s}{2\sqrt n}} - \frac {u}{4n}}^n \to e^{(s^2-u)/4}.
\]

Since $(n_k-n_{k-1})/n\to x_k-x_{k-1}$, from \eqref{eq:Gxcn} we get
\begin{equation*}\label{eq:Gxc}
\limn G_{\vvx,\vvc,n}(\vv u) =\exp\pp{\frac14\summ k1{d}(s_k^2-u_k)(x_k-x_{k-1})}=: G_{\vvx,\vvc}(\vv u)
\end{equation*}
 for $\vv u \in [0,\infty)^d$.

Adapting  \cite[proof of Proposition 4.7, (4.24)]{bryc19limit} we also have
\begin{equation}\label{eq:G upper}
G_{\vvx,\vvc,n}(\vvu) \le K \prodd k1d\exp\pp{-\frac{n_k-n_{k-1}}{4n}u_k} \mbox{ for all } \vv u \in\RR_+^d ,
\end{equation}
where  $K$   is a constant independent of $n$.
We will use the following fact which follows from
\cite[Theorem 5.5]{billingsley1968}.

\begin{lemma} %
Suppose a sequence  $\mu\topp n$ of probability measures on  $\RR_+^d$  converges weakly, $\mu\topp n\Rightarrow\mu$. Let  $G_{n}$ be a   sequence of uniformly bounded measurable functions, $G_n:\RR_+^d\to [-K,K]$ for some $K>0$,   such that for all $\vvu_n\to \vv u\in\R_+^d$ we have
\[\limn G_{n}(\vv u_n) = G_{}(\vv u).%
\]
Then  we have
\begin{equation}\label{Bil:Ex6.6}
\limn \int G_n(\vv u)\mu\topp  n(d \vv u)= \int G(\vv u)\mu(d \vv u).
\end{equation}
\end{lemma}
\arxivc{\begin{proof}
By \cite[Theorem 5.5]{billingsley1968}, for every bounded continuous function $f$ on $\RR$ we have
\[\int f(G_n(\vv u))\mu_n(d \vv u)\to \int f(G(\vv u))\mu(d \vv u).\]
Now take bounded continuous function such that $f(x)=x$ on $[-K,K]$.
\end{proof}}

From \eqref{eq:G upper} we see that functions $ G_{\vvx,\vvc,n}$ are uniformly bounded in $\vv u\in\RR_+^d$. We also note that
\[
\limn G_{\vvx,\vvc,n}(\vv u_n) = G_{\vvx,\vvc}(\vv u), \quad\mbox{ if } \vvu_n\to \vv u\in\R_+^d \mmas n\to\infty,
\]
see \cite[(4.25)]{bryc19limit}.

We now apply \eqref{Bil:Ex6.6}  to a sequence of probability measures
\[\mu_{u_1}\topp n(d u_2,\dots,d u_d)=\prodd k1{d-1}\wt p_{-s_k,-s_{k+1}}\topp n(u_k,u_{k+1})du_2\dots du_d\]
with fixed $u_1>0$, which by Scheff\'e's Lemma converges to the measure $\mu$ constructed from the transition densities \eqref{eq:pst}.
Recall that we have established the convergence of transition densities $\wt p_{s,t}\topp n$ in  Lemma \ref{lem:2}.

It then follows that $\mfa u_1>0$ we have
\begin{align}
\mathsf G_n(u_1) &:=  \esp\pp{G_{\vvx,\vvc,n}(u_1,\wt Y_{-s_2}\topp n,\dots,\wt Y_{-s_d}\topp n)\mmid \wt Y_{-s_1}\topp n = u_1}
 \nonumber\\
 &  \to\mathsf G(u_1) := %
 \begin{cases}
\displaystyle \int_{\R_+^{d-1}}G_{\vvx,\vvc}(u_1,u_2,\dots,u_d)\frac{(\A+s_1)^2+u_1}{(\A+s_d)^2+u_d}\prodd k1{d-1}\p_{s_k-s_{k+1}}(u_k,u_{k+1})du_2\dots du_d, &\A<\infty, \\\\
\displaystyle \int_{\R_+^{d-1}}G_{\vvx,\vvc}(u_1,u_2,\dots,u_d)\prodd k1{d-1}\p_{s_k-s_{k+1}}(u_k,u_{k+1})du_2\dots du_d, &\A=\infty,\label{eq:Gn}
 \end{cases}
\end{align}
 as $n\to\infty$.
  To prove \eqref{eq:num_limit} we analyze
\begin{equation*}
  \label{eq:Gn1}
\limn \frac1{\mathsf \Pi_n}\int_{\RR_+}\ind_{u\in[0,2n]}\mathsf G_n(u) \wt \pi\topp n_{-s_1}(u) \,du.
\end{equation*}
  In view of bounds \eqref{eq:pi upper} and \eqref{eq:G upper},
  we can apply the  dominated convergence theorem.
From pointwise convergence of $u\mapsto \ind_{u\in[0,2n]} G_n(u)$,  and $\wt \pi\topp n_t/\mathsf \Pi_n$, see \eqref{eq:Gn} and \eqref{eq:pi}, we get
\begin{multline*}
\limn \frac1{\mathsf \Pi_n}\int_{\RR_+}\ind_{u\in[0,2n]}\mathsf G_n(u) \wt \pi\topp n_{-s_1}(u) du
=
\begin{cases}
\displaystyle \int_{\R_+}\mathsf G(u)
  \frac{\sqrt u}{((\A+s_1)^2+u)((\C-s_1)^2+u)}\,du, &  \mbox{ if } \A,\C\in\RR, \\ \\
\displaystyle   \int_{\R_+}\mathsf G(u)
 \frac{\sqrt u}{(\A+s_1)^2+u}\,du, &  \mbox{ if } \A\in\RR,\; \C=\infty, \\ \\
\displaystyle   \int_{\R_+}\mathsf G(u)
   \sqrt u \,du, &  \mbox{ if } \A=\C=\infty. \\
\end{cases}
\end{multline*}
Recalling the definition  \eqref{eq:Gn} of $\mathsf G(u)$, this   ends the proof.
\end{proof}

Next we analyze the integral $I_{n,1}$ introduced in \eqref{eq:I1+I2}.  %
We observe that for a fixed $\delta\in(0,1)$ and large enough $n$ we have $\cosh(s_k/\sqrt n)\leq 1+\tfrac{s_1^2}{n}<1+\delta$. So with  $Y\topp n_{t_{1,n}}\in(-1-\delta,0)$,
\[
-\delta=1-(1+\delta)<\cosh(s_1/\sqrt n) +Y\topp n_{t_{1,n}}<
1+s_1^2/n<1+\delta.
\]
Consequently, recalling that $\cosh(s_k/\sqrt{n})\leq 1+s_k^2/n\leq 1+s_1^2/n$ if $s_1^2/n<1$,
\begin{equation}\label{I1*}
|I_{n,1}|\leq
(2+\tfrac{s_1^2}{n})^{n-n_1}(
1+\delta)^{n_1}\leq e^{s_1^2/2} 2^n \left(\tfrac{1+\delta}{2}\right)^{n x_1}.
\end{equation}
Referring to \eqref{eq:Pi-bd}, we see that
\begin{equation*}
    \label{eq:I1-negl}
    \limn \frac{I_{n,1}}{2^n \mathsf \Pi_n}=0.
\end{equation*}
\arxivc{Indeed, given $\eps>0$, we use  \eqref{eq:Pi-bd} to choose $N$ such that for $n>N$
$$\frac{1}{\mathsf \Pi_n}\leq e^{\eps n}.$$
Then by \eqref{I1*} we have
$$
 \frac{I_{n,1}}{2^n \mathsf \Pi_n} \leq e^{\eps n} e^{s_1^2/2} ((1+\delta)/2)^{n x_1} \to 0.
$$
provided $\eps$ is small enough so that $x_1\log ((1+\delta/2))>-\eps$.
}
Combining the asymptotics of $I_{n,1}$ and $I_{n,2}$ we see that
\begin{equation}
\label{eq:licznik}
 \limn \frac{\esp\bb{\prodd k1{d} \pp{\cosh(s_k/\sqrt n)+Y\topp n_{e^{-2s_k/\sqrt n}}}^{n_k-n_{k-1}}}}{2^n \mathsf \Pi_n}
\end{equation}
is given by the right hand side of \eqref{eq:num_limit}.

\subsubsection{Asymptotics of the denominator in \protect\eqref{eq:Laplace}}
\begin{lemma}\label{lem:4}
Under the  Assumption \ref{assump:1}, we have %
\begin{equation}\label{eq:denom}
 \limn\frac1{2^n\mathsf \Pi_n}\esp \pp{1+Y_1\topp n}^n
= \begin{cases}
\displaystyle \frac\pi{\sqrt 2}\, \mathfrak C_{\A,\C}, & \mbox{ if } \A,\C\in\RR, \\ \\
  2\sqrt{2}\,\pi\, \mathfrak C_{\A,\infty}, & \mbox{ if } \A\in\R, \C = \infty, \\ \\
4{\sqrt \pi},&  \mbox{ if } \A = \C = \infty,
\end{cases}
\end{equation}
where $\mathfrak C_{\A,\C}$ is defined in \eqref{eq:C_ac}
\end{lemma}
\begin{proof}
Since $\C>0$ and $B_n,D_n\in(-1,0]$, the  atoms for $Y_1\topp n$ may only arise from $A_n$ if $\A<0$.

We first consider the case  $\A\geq 0,\C>0$. The proof is similar as in \cite[Lemma 4.5]{bryc19limit}.
Since $Y_1\topp n$ has a density supported on $[-1,1]$, we have
\[
  0\leq  \esp\pp{1+Y_1\topp n}^n-  \esp\left[\ind_{Y_1\topp n\in[0,1]}\pp{1+Y_1\topp n}^n \right] =\int_{-1}^0(1+u)^n \pi_1\topp n(u)du  \leq 1.
\]
Therefore by \eqref{eq:Pi-bd} %
and  \eqref{eq:Pi_n}   we get
\begin{align*}
\mathsf \Pi_n^{-1} 2^{-n}\esp\pp{1+Y_1\topp n}^n
&
 \sim  \mathsf \Pi_n^{-1}2^{-n} \esp\left[\ind_{Y_1\topp n\in[0,1]}\pp{1+Y_1\topp n}^n \right]
= \mathsf \Pi_n^{-1}
\esp\left[\ind_{\wt Y_0\topp n\in[0,2n]}\pp{1-\tfrac{\wt Y_0\topp n}{4n}}^n\right]
\\
& \sim
\begin{cases}
   \int_{\R_+}\frac{\sqrt u}{(\A^2+u)(\C^2+u)}e^{-u/4}du, &  0\leq \A<\infty,  0<\C<\infty, \\ \\
  \int_{\R_+}\frac{\sqrt u}{\A^2+u }e^{-u/4}du, & 0\leq \A<\infty, \C=\infty, \\ \\
  \int_{\R_+} \sqrt u e^{-u/4}du, & \A=\C=\infty.
\end{cases}
\label{eq:first case}
\end{align*}
see \eqref{eq:num_limit}  and \eqref{PsiAC} with $d=1$ and $s_1=0$, compare \eqref{**G**}. %

Recall \eqref{eq:H}.
To evaluate the integrals, we use the identity
\begin{equation}\label{eq:Laplace_H}
\int_0^\infty \frac{\sqrt u}{\A^2+u}e^{-u/4}du = \pi\pp{\frac2{\sqrt \pi} - |\A| H(|\A|/2)},
\end{equation}
which is valid for all real $\A$. (This is \cite[Section 4.2 formula (22) pg 136]{erdelyi1954fg}  evaluated at  $p=1/4$.)

For $\A\geq 0,\C>0, \A\ne\C$, we can apply \eqref{eq:Laplace_H} directly. We get
\begin{align*}
\int_{\R_+} \frac{\sqrt u}{(\A^2+u)(\C^2+u)}e^{-u/4}du & = \int_{\R_+}\pp{\frac1{\A^2+u}-\frac1{\C^2+u}}\frac{\sqrt u}{\C^2-\A^2}e^{-u/4}du
 = \pi \frac{\C H(\C/2) - \A H(\A/2)}{\C^2-\A^2},
\end{align*}
which proves \eqref{eq:denom} for this case. To prove \eqref{eq:denom}
for $\A\geq 0,\C = \infty$ and $\A = \C = \infty$, we   apply the identity \eqref{eq:Laplace_H} and $\int_0^\infty e^{-u/4}\sqrt udu = 4\sqrt \pi$ respectively.

For $\A = \C>0$, we use instead the identity
\[
\int_{\R_+}\frac{\sqrt u}{(\A^2+u)^2}e^{-u/4}du = \frac\pi{2\A} \pp{\pp{1+\frac{\A^2}2} H(\A/2) - \frac\A{\sqrt\pi}},
\]
which can be obtained by differentiating \eqref{eq:Laplace_H} with respect to   $\A$,
and we get \eqref{eq:denom}.
This completes the proof for the case $\A\geq 0,\C>0$.

Next, we consider $\A<0<\C$ and $\A+\C>0$.
 The law of $Y\topp n_1$ now has both atomic and continuous part, and for the continuous part, restricted to $[-1,1]$, we   use \eqref{eq:Laplace_H} for $\A<0$. Namely, we get
\begin{equation}\label{eq:denom_cont}
 2^{-n}\esp\pp{\pp{1+Y\topp n_1}^n\inddd{Y_1\topp n\in[-1,1]}} \sim
 \pi \mathsf \Pi_n \cdot \begin{cases}
 \displaystyle \frac{\C H(\C/2) +\A H(-\A/2)}{\C^2-\A^2} & \A<0, \C<\infty,\\
  \\
\displaystyle \pp{\frac2{\sqrt \pi} + \A H(-\A/2)} & \A<0, \C=\infty.
 \end{cases}
\end{equation}
It remains to compute the atomic part. Again, for $n$ large enough there is a single atom at
\begin{equation*}
    \label{eq:atomAy}
    \mathsf y_n:=\frac12\pp{A_n+\frac1{A_n}},
\end{equation*}
and   its probability is
\begin{align*}\label{eq:atomAp}
\mathfrak p_n=\frac{\qp{A_n^{-2},B_nC_n,B_nD_n,C_nD_n}}{\qp{B_n/A_n,C_n/A_n,D_n/A_n,A_nB_nC_nD_n}}& \sim
\frac{(1-A_n^{-2})}{(1-C_n/A_n)}\cdot \frac{1-B_nD_n}{1-A_nB_nC_nD_n} \\
&\sim \frac{2(A_n-1)}{A_n-C_n}\cdot \frac{1-B_nD_n}{1-A_nB_nC_nD_n}, \nonumber
\end{align*}
see \eqref{1*1}.

We note that %
\begin{equation*}
  \label{eq:atomA}
  2^{-n}(1+\mathsf y_n)^n = \frac1{2^{2n}}\pp{A_n+\frac1{A_n}+2}^n = \pp{1+\frac1{4A_n}(A_n-1)^2}^n\sim e^{\A^2/4}.
\end{equation*}

We also note that %
\begin{equation*}%
\frac{(1-A_n^{-2})}{(1-C_n/A_n)} \sim \frac{2(A_n-1)}{A_n-C_n} \sim
\begin{cases}
\displaystyle  \frac{-2\A}{\C-\A} &  \A<0,\;0<\C<\infty,\\\\
\displaystyle  \frac{-2\A}{\sqrt{n}(1-C_n)} & \A<0,\;\C=\infty.
\end{cases}
\end{equation*}
\arxivc{Indeed,  $(A_n-C_n)/(1-C_n)=1+(A_n-1)/(C_n-1)\sim 1$ if $\sqrt{n}(1-C_n)\to \infty$}
It follows that
\begin{align*}
2^{-n}\esp & \pp{\pp{1+Y\topp n_1}^n\inddd{Y_1\topp n\notin[-1,1]}}  \sim \frac{1-B_nD_n}{1-A_nB_nC_nD_n} e^{\A^2/4}\times
\begin{cases}
\displaystyle  \frac{-2\A}{\C-\A} &  \A<0,\;0<\C<\infty,\\\\
\displaystyle  \frac{-2\A}{\sqrt{n}(1-C_n)} & \A<0,\;\C=\infty.
\end{cases}
\end{align*}
In view of \eqref{eq:Pi_n} we obtain
\begin{align}\label{eq:denom_atom}
2^{-n}\esp\pp{\pp{1+Y\topp n_1}^n\inddd{Y_1\topp n\notin[-1,1]}}
& \sim \pi\, \mathsf \Pi_n \; e^{\A^2/4}
\times \begin{cases}
\displaystyle \frac{-2\A}{\C^2-\A^2}  &  \A<0,\;0<\C<\infty,\\\\
\displaystyle  -2\A  & \A<0,\;\C=\infty.
\end{cases}
\end{align}

Combining \eqref{eq:denom_cont} and \eqref{eq:denom_atom}, we have
\begin{align*}
 2^{-n}\esp\pp{1+Y\topp n_1}^n & \sim \pi\,\mathsf\Pi_n  \times
 \begin{cases}
\displaystyle  \frac{\C H(\C/2) -\A H(\A/2)}{\C^2-\A^2}  &  \A<0,\;0<\C<\infty,\\\\
\displaystyle  \pp
 { \frac2{\sqrt \pi}-\A H(\A/2)}  & \A<0,\;\C=\infty,
 \end{cases}
\end{align*}
where we used the identity $H(x)+H(-x)=2\exp(x^2)$.
\end{proof}

\begin{proof}[Proof of Theorem \ref{thm:2}]
It remains to put all pieces together.
To prove \eqref{konkluzja}, we refer to \eqref{eq:Laplace}.
The normalizing constant in $\Psi_{\vv x}\topp{\A,\C}(\vv c)$ arises from the asymptotics of the denominator in  \eqref{eq:Laplace} given in
 Lemma \ref{lem:4}.
The remaining part of \eqref{konkluzja} is  a consequence of the  asymptotics of the numerator in  \eqref{eq:Laplace}
which is given in \eqref{eq:licznik} and \eqref{eq:num_limit}.
In particular, %
the first factor on the right hand side of \eqref{eq:num_limit}
gives  the fluctuation part corresponding to the Brownian motion:
\begin{equation*}\label{eq:1/4}
 \exp\pp{\frac14\summ k1d s_k^2(x_k-x_{k-1})} = \esp \exp\pp{-\summ k1d \frac{c_k}{\sqrt 2}\BB_{x_k}}.
\end{equation*}
\end{proof}

\subsection{A dual representation for the Laplace transform}
\label{sec:duality}
Informally, a dual representation for the Laplace transform expresses the Laplace transform of the finite-dimensional distributions of a process as a Laplace transform of another process by exchanging the roles of the arguments of the Laplace transform and the time variables for the two processes.
We shall establish
 a general dual representation identity
that applies to the three limit Laplace transforms in Theorem \ref{thm:2}, see \eqref{PsiAC}.
Recall that %
$\p_{t-s}(x,y)$ is the transition probability density function \eqref{eq:Biane p} of   the   $1/2$-stable Biane process, and $\g_t(x,y)$ is the transition sub-probability density function \eqref{eq:g} of the Brownian motion killed at hitting zero.

\begin{proposition}\label{prop:duality} Let $f,g$ be two measurable functions on $\RR_+$.
With $c_1,\cdots,c_{d-1}>0$ and $0= x_0<x_1<\cdots<x_d$,
\begin{multline}\label{eq:duality}
\int_{\R_+^d}e^{-\frac14\summ k1d(x_k-x_{k-1})u_k}f(u_1)\left(\prodd k1{d-1}\p_{c_k}(u_k,u_{k+1})\right)g(u_d) %
d\vv u
\\
= \frac8\pi\int_{\R_+^{d-1}}e^{-\frac1{\sqrt 2}\sum\limits_{k=1}^{d-1}c_k z_k}\what f(z_1)\left(\prodd k2{d-1}\g_{x_k-x_{k-1}}(z_{k-1},z_k)\right)\what g(z_{d-1})d\vv z,
\end{multline}
where %
\begin{align*}
\what f(z) & := \int_{\R_+}f(2u^2)\sin(uz)e^{-x_1u^2/2}du,\\
\what g(z)&:= \int_{\R_+}g(2u^2)u\sin(uz)e^{-(x_d-x_{d-1})u^2/2}du,
\end{align*}
provided that  the functions under the multiple integrals  in \eqref{eq:duality} are absolutely integrable.
\end{proposition}
The proof follows the same idea as in \cite{bryc18dual}, where the cases $\A = 0,\C = \infty$ (Brownian excursion) and  $\A =\C = \infty$ (Brownian excursion) have been established therein.
The novelty is that now new expressions of $f$ and $g$ show up in applications and they lead to new %
dual representations of the Laplace transforms
in Proposition \ref{prop:duality summary}.
Another dual representation formula can be found in \cite{bryc21markov}.

\begin{proof}
 With a change of variables $u_k\mapsto u_k^2$, the left-hand side of \eqref{eq:duality} becomes
\begin{align}
2^d\int_{\R_+^d}e^{{-\frac14\summ k1d(x_k-x_{k-1})u_k^2}}u_1f(u_1^2)\left(\prodd k1{d-1}u_{k+1}\p_{c_k}(u_k^2,u_{k+1}^2)\right)g(u_d^2)d\vv u.\label{eq:duality1}
\end{align}
Recall that \cite[after Eq.~(2.2)]{bryc18dual} %
\begin{equation*}\label{eq:ypc}
y\p_t(x^2,y^2) = \frac1\pi \frac yx\int_{\R_+} e^{-tz}\sin(xz)\sin(yz)dz,
\end{equation*}
and
\[
\int_{\R_+} e^{-tx^2/2}\sin(xy_1)\sin (xy_2)dx = \frac\pi2\g_t(y_1,y_2).
\]
Then,
\eqref{eq:duality1} becomes
 \begin{align*}
&  \frac{2^d}{\pi^{d-1}}  \int_{\R_+^d}u_1f(u_1^2) \pp{\prodd k1{d-1}\frac{u_{k+1}}{u_k}\int_{\R_+}e^{-c_kz_k}\sin(u_kz_k)\sin(u_{k+1}z_k)dz_k} g(u_d^2)\; e^{-\frac14\sum\limits_{k=1}^d(x_k-x_{k-1})u_k^2}d\vv u\\
&  = \frac 4\pi\int_{\R_+^{d-1}}e^{-\sum\limits_{k=1}^{d-1}c_kz_k}\Bigg(\int_{\R_+}f(u_1^2)\sin(u_1z_1)e^{-x_1u_1^2/4}du_1   \prodd k2{d-1}\frac2\pi\int_{\R_+}e^{-(x_k-x_{k-1})u_k^2/4}\sin(u_kz_k)\sin(u_kz_{k-1})du_k \\
&  \quad\times \int_{\R_+}u_d e^{-(x_d-x_{d-1})u_d^2/4}\sin(u_dz_{d-1})g(u_d^2)du_d \Bigg)d\vv z.
\end{align*}
One can then show that  the above is the same as
\begin{align*}
\frac4\pi\int_{\R_+^{d-1}}e^{-\sum\limits_{k=1}^{d-1}c_kz_k}(\sqrt 2\what f(\sqrt 2z_1))\prodd k2{d-1}\g_{(x_k-x_{k-1})/2}(z_{k-1},z_k)(2\what g(\sqrt 2z_{d-1}))d\vv z.
 \end{align*}
The desired result now follows from $\g_{x/2}(z,z') = \sqrt 2 \g_x(\sqrt 2z,\sqrt 2z')$ and changes of variables $\sqrt 2z_k\mapsto z_k, k=1,\dots,d-1$.
\end{proof}
\begin{proof}[Proof of Proposition \ref{prop:duality summary}]
We consider separately the three cases in \eqref{PsiAC}, starting with the easiest.
\medskip

\noindent (i)    Case  $\A = \infty, \C = \infty$: the goal is to show
\[
\Psi\topp{\infty,\infty}_\vvx(\vvc) %
=
\esp \exp\pp{-\frac1{\sqrt 2}\sum\limits_{k=1}^{d-1}c_k\BB^{\rm ex}_{x_k}}.
\]

We apply Proposition \ref{prop:duality}  with  $f(u) = \sqrt u$, $g(u) = 1$,  $x_d = 1$.  Then by straightforward calculation  we get,
\begin{align}
\what f(z)  & = \sqrt 2\int_{\R_+}ue^{-x_1u^2/2}\sin(uz)du = \sqrt 2\pi \ell_{x_1}(z),\label{eq:whatf(z)}\\
\what g(z) & = \int_{\R_+}ue^{-(1-x_{d-1})u^2/2}\sin(uz)du = \pi \ell_{1-x_{d-1}}(z),\label{eq:whatg(z)}
\end{align}
see \eqref{eq:ell}.
Looking at the right hand side of \eqref{eq:duality} we then recognize
\[
\sqrt {8\pi} \ell_{x_1}(z_1)\prodd k2{d-1}\g_{x_{k}-x_{k-1}}(z_{k-1},z_k)\ell_{1-x_{d-1}}(z_{d-1}), \quad (z_1,\dots,z_{d-1})\in\R_+^{d-1}, %
\]
as the joint probability density of a Brownian excursion at times $x_1,\dots,x_{d-1}$, see \eqref{eq:pdf_Bex}.
The desired identity now follows. %

\medskip

\noindent (ii) %
 Case $\A\in\RR, \C = \infty$:    the goal is to show
 \begin{align*}
\Psi\topp{\A,\infty}_\vvx(\vvc)
& = \esp\exp\pp{-\frac1{\sqrt 2}\summ k1d c_k\eta\topp{\A,\infty}_{x_k}}.
\end{align*}
This time, consider $f(u) = \sqrt u$ as before (so $\what f(z)$ is in \eqref{eq:whatf(z)}), but
\begin{equation}\label{eq:g(u)}
g(u) = \frac1{\alpha^2+u}, \quad \alpha= \A+s_d = \A+c_d\ge 0.
\end{equation}
Then, %
 \begin{align*}
 \what g(z) & = \int_{\R_+}e^{-(1-x_{d-1})u^2/2}\frac{u}{\alpha^2+2u^2} \sin(uz)du \\
 & = \frac14\int_0^\infty e^{-(1-x_{d-1})u/2}\frac {\sin(\sqrt uz)}{\alpha^2/2+u}du = \frac14 \mathsf F\pp{z\mmid \frac \alpha{\sqrt 2},\frac{1-x_{d-1}}2}.
 \end{align*}
with %
\begin{align}
\mathsf F(z\mid \alpha,\theta) & :=\int_0^\infty e^{-\theta u}\frac{\sin(z \sqrt u)}{\alpha^2+u}du
 = \frac \pi 2e^{\alpha^2\theta}\pp{e^{-\alpha z}\erfc \pp{\alpha\sqrt{\theta }-\frac z{2\sqrt\theta}}
- e^{\alpha z}\erfc \pp{\alpha\sqrt{\theta}+\frac z{2\sqrt\theta}}}\nonumber\\
\label{eq:FH}
&  =   \frac\pi 2  e^{-z^2/(4\theta)} \pp{H\pp{\alpha\sqrt\theta - \frac z{2\sqrt\theta}} - H\pp{\alpha\sqrt\theta + \frac z{2\sqrt\theta}}}   =  \pi \int_0^\infty \g_{2\theta}(z,y)e^{-\alpha y}dy,
\end{align}
where the second step follows from \cite[2.4.3.24]{prudnikov92integrals}, and the last from  Lemma \ref{lem:FH}. Therefore,
\begin{equation}\label{eq:ghat}
\what g(z_{d-1}) = \frac\pi4\int_0^\infty \g_{1-x_{d-1}}(z_{d-1},z_d)e^{-\alpha z_d/\sqrt 2}dz_d.
\end{equation}
Now, %
\eqref{eq:duality} becomes %
\begin{align*}
\Psi_{\vv x}\topp{\A,\infty}(\vv c)
& = \frac 8{\pi^2 \sqrt{8}\mathfrak C_{\A,\infty}}\int_{\R_+^{d-1}} \what f(z_1)\prodd k2{d-1}\g_{x_k-x_{k-1}}(z_{k-1},z_k) \what g(z_{d-1})e^{-\frac1{\sqrt 2}\sum\limits_{k=1}^{d-1}c_kz_k}d\vv z\\
& = \frac{1}{  \mathfrak C_{\A,\infty}}\int_{\R_+^{d}}  \ell_{x_1}(z_1)\prodd k2{d}\g_{x_k-x_{k-1}}(z_{k-1},z_k)e^{-\frac1{\sqrt 2}\sum\limits_{k=1}^{d} c_kz_k}e^{-\A z_d/\sqrt 2}d\vv z.
\end{align*}
To recognize the above as the desired Laplace transform,
we invoke \eqref{eq:proces-Ainfty}.

\medskip
\noindent (iii) %
 Case $\A,\C\in\RR, \A+\C>0$:
the goal is to show
\begin{align*}
\Psi\topp{\A,\C}(\vv c)
& = \esp \exp\pp{-\frac1{\sqrt 2}\summ k1d c_k\pp{\wt\eta\topp{\A,\C}_{x_k}-\wt\eta\topp{\A,\C}_0}}.
\end{align*}
We apply Proposition \ref{prop:duality} with $x_d=1$.
This time, consider
$
f(u) = {\sqrt u}/({\gamma^2+u})$, where $\gamma= \C-s_1=\C-c_1-\dots -c_d\ge 0$,
and we take
$g(u)$ as in \eqref{eq:g(u)}.
We have, by the same calculation for $\what g(z)$
as
before,   see \eqref{eq:FH},
\begin{align}
\what f(z) & = \int_{\R_+}\frac {\sqrt 2u}{\gamma^2+2u^2}\sin(uz)e^{-x_1u^2/2}du = \frac1{2\sqrt 2}\mathsf F\pp{z\mmid \frac \gamma{\sqrt 2},\frac{x_1}2}
= \frac\pi{2\sqrt 2}\int_{\R_+}\g_{x_1}(z_0,z_1)e^{-\gamma z_0/\sqrt 2}dz_0.\label{eq:whatf(z)2}
\end{align}
By \eqref{eq:ghat} and \eqref{eq:whatf(z)2}, the duality \eqref{eq:duality} gives
\[
\Psi\topp{\A,\C}_\vvx(\vvc)  = \frac1{ \mathfrak C_{\A,\C}} \int_{\R_+^{d+1}}e^{-\C z_0/\sqrt 2-\A z_d/\sqrt 2}\prodd k1 d\g_{x_k-x_{k-1}}(z_{k-1},z_k)e^{-\frac1{\sqrt 2}\summ k1dc_k(z_k-z_0)}d\vv z.
\]
In view of \eqref{eq:eta} the desired result now follows.
\end{proof}

\subsection*{Acknowledgement} %
The authors thank Ivan Corwin and Alisa Knizel for sharing an early version of \cite{corwin21stationary}. We thank Guillaume Barraquand for information about Ref. \cite{barraquand2022steady} and the discussion of its contents. We also thank Alexey Kuznetsov for several discussions that inspired Proposition \ref{prop:duality}. %

WEB's research was partially supported by Simons Foundation/SFARI Award Number: 703475, US.
YW's research was partially supported by Army Research Office, US (W911NF-20-1-0139).
JW's research was partially supported by grant  IDUB no. 1820/366/201/2021, Poland.

\appendix
\section{Auxiliary integrals}
Recall $\g_t$ as in \eqref{eq:g}. Recall $H(x):=e^{x^2}\erfc(x)$ in \eqref{eq:H}.
\begin{lemma}\label{lem:FH}
For all $c\in\R$, %
\[
\int_0^\infty \g_t(x,y)e^{-c y}dy = \frac12 e^{-x^2/(2t)} \pp{H\pp{\frac{c t-x}{\sqrt{2t}}} - H\pp{\frac{c t+x}{\sqrt{2t}}}}.
\]
\end{lemma}
\begin{proof}
Write
\[
\int_0^\infty \g_t(x,y)e^{-c y}dy  = \int_0^\infty \frac1{\sqrt{2\pi t}}\pp{\exp\pp{-\frac{(x-y)^2}{2t}}- \exp\pp{-\frac{(x+y)^2}{2t}}}e^{-c y}dy.
\]
First we compute %
 \begin{align*}
 \int_0^\infty \frac1{\sqrt{2\pi t}}e^{-(y-x)^2/(2t)}e^{- c y}dy & = e^{(x-c t)^2/(2t)} e^{-x^2/(2t)}\int_0^\infty \frac1{\sqrt{2\pi t}}e^{-(y-(x-c t))^2/(2t)}dy \\
 & = \frac12 H\pp{\frac{c t-x}{\sqrt{2t}}}e^{-x^2/(2t)}.
 \end{align*}
The second term is obtained by replacing $x$ by $-x$.
\end{proof}
\begin{lemma}\label{lem:C_ac}
For $a,c\in\R$ and $a+c>0$, %
\begin{equation*}\label{eq:Ctab}
\int_0^\infty \int_0^\infty \g_t(x,y)e^{-a x-c y}dxdy = \begin{cases}
\displaystyle \frac{a H(a \sqrt{t/2})-c H(c \sqrt{t/2})}{a^2-c^2}, &\mbox{ if } a\ne c,
\\\\
\displaystyle \frac{1+a^2te^{a^2t/2}}{2a} H(a\sqrt {t/2}) - \frac t{\sqrt{2\pi t}}, & \mbox{ if } a=c.
\end{cases}
\end{equation*}
\end{lemma}
\begin{proof}
Assume $a\ne c$ first. We first compute
\begin{equation}\label{eq:x+y}
\frac1{\sqrt{2\pi t}}\int_0^\infty\int_0^\infty  e^{-(x+y)^2/(2t) - a x-c y}dxdy = \frac1{2(a-c)}\pp{H(c\sqrt{t/2}) - H(a\sqrt{t/2})}.
\end{equation}
By change of variables $x = \theta u, y = (1-\theta) u$, %
the double integral becomes
\[
  \int_0^\infty u e^{-u^2/(2t)-c u}\int_0^1e^{-\theta u(a-c )}d\theta du
  = \frac1{a-c}\pp{\int_0^\infty e^{-u^2/(2t)-c u}du - \int_0^\infty e^{-u^2/(2t)-a u}du}.
\]
Since
\begin{equation}\label{eq:Hc}
\frac1{\sqrt{2\pi t}}\int_0^\infty e^{-u^2/(2t)-c u}du =
\frac{e^{c^2t/2}}{\sqrt{2\pi t}}\int_0^\infty e^{-(u+c t)^2/(2t)}du  = \frac 12 H\pp{c\sqrt{t/2}},
\end{equation}
it follows that \eqref{eq:x+y} holds.
Next, we show\begin{equation}\label{eq:x-y}
\frac1{\sqrt{2\pi t}}\int_0^\infty\int_0^\infty  e^{-(x-y)^2/(2t) - a x-c y}dxdy = \frac1{2(a +c)}\pp{H(c\sqrt{t/2}) + H(a\sqrt{t/2})}.
\end{equation}
This time, first consider the region $\mathcal{U}=\{(x,y)\in\R_+^2, x>y\}$, and for this region consider $x = u\theta, y = u(\theta-1)$. Then
 \begin{align*}
 \iint_{\mathcal{U}}
 e^{-(x-y)^2/(2t)-a x-c y}dxdy & = \int_0^\infty ue^{-u^2/(2t)+c u}\int_1^\infty e^{-(a +c)u\theta}d\theta du \\
 & = \frac1{a+c}\int_0^\infty e^{-u^2/(2t)-a u}du.
 \end{align*}
A similar calculation holds for the region $\{(x,y)\in\R_+^2, x<y\}$. Then, combining with \eqref{eq:Hc} we obtain \eqref{eq:x-y}. The desired result follows from \eqref{eq:x+y} and \eqref{eq:x-y}.

Next, assume $a= c$. Then \eqref{eq:x+y} becomes
\begin{align*}
\frac1{\sqrt {2\pi t}}\int_0^\infty ue^{-u^2/(2t)-a u}du
& = \frac{e^{a^2t/2}}{\sqrt{2\pi t}}\int_0^\infty u e^{-(u+a t)^2/(2t)}du
\\
& = \frac{e^{a^2t/2}}{\sqrt{2\pi t}}\pp{\int_{a t}^\infty ue^{-u^2/(2t)}du -a t \int_0^\infty e^{-(u+a t)^2/(2t)}du }\\
&
= \frac t{\sqrt{2\pi t}} - a t e^{a^2t/2} \frac 12 H(a\sqrt{t/2}).
\end{align*}
This time, by the above and \eqref{eq:x-y}, we have the desired result for $a=c$.
\end{proof}

\section{Asymptotics of Pochhammer symbols}

\begin{lemma} Fix $0\leq q<1$.
   For complex numbers $\alpha_n\in\mathbb C$ such that $\alpha_n\to \alpha$ with $|q\alpha|<1$ we have
\begin{equation}\label{1*1}
 (\alpha_n;q)_\infty =z_n(1-\alpha_n)(\alpha q;q)_\infty \mbox{ for some complex sequence } z_n\to 1.
\end{equation}
Furthermore, if $|\alpha_n |\leq 1$ then
\begin{equation}\label{2*2}
 (q;q)_\infty\; |1-\alpha_n| \;\leq |(\alpha_n;q)_\infty|\leq (-q;q)_\infty \; |1-\alpha_n| \;.
\end{equation}
\end{lemma}

\begin{proof}%
It is well known  \cite{gaspar1990basic} that function
\[
 z\mapsto (z;q)_\infty=\sum_{k=0}^\infty (-z)^k q^{k(k-1)/2}/(q;q)_k
 \] is analytic if $|z|<1$.
Therefore,
\[(\alpha_n;q)_\infty =(1-\alpha_n)(\alpha_n q ;q)_\infty =(1-\alpha_n)(\alpha q ;q)_\infty z_n,\]
where
\[
z_n=\frac{(\alpha_n q ;q)_\infty}{(\alpha q ;q)_\infty}\to 1 \mbox{ as $n\to \infty$}.
\]

For the second part of the proof, since  $|\alpha_n|\leq 1$ and $q\geq 0$ we have
\begin{align*}
   |(\alpha_n;q)_\infty| & =|1-\alpha_n||1-\alpha_n q|\cdots|1-\alpha_nq^k|\cdots
     \geq
|1-\alpha_n|(1-|\alpha_n q|)\cdots(1-|\alpha_nq^k|)\cdots
\\
&\geq  |1-\alpha_n|(1-  q)\cdots(1- q^k)\cdots=|1-\alpha_n|(q;q)_\infty.
\end{align*}
Similarly,
$|(\alpha_n;q)_\infty|\leq |1-\alpha_n|  \;(1+ q)\cdots(1+ q^k)\cdots=|1-\alpha_n|(-q;q)_\infty
$.
\end{proof}

\def\polhk#1{\setbox0=\hbox{#1}{\ooalign{\hidewidth
  \lower1.5ex\hbox{`}\hidewidth\crcr\unhbox0}}}


\begin{thebibliography}{10}

\bibitem{askey85some}
Richard Askey and James Wilson.
\newblock Some basic hypergeometric orthogonal polynomials that generalize
  {J}acobi polynomials.
\newblock {\em Mem. Amer. Math. Soc.}, 54(319):iv--55, 1985.

\bibitem{barraquand2022steady}
Guillaume Barraquand and Pierre {Le Doussal}.
\newblock Steady state of the {KPZ} equation on an interval and {L}iouville
  quantum mechanics.
\newblock {\em Europhysics Letters}, 137(6):61003, 2022.
\newblock ArXiv preprint with Supplementary material:
  \url{https://arxiv.org/abs/2105.15178}.

\bibitem{billingsley1968}
Patrick Billingsley.
\newblock {\em Convergence of probability measures}.
\newblock John Wiley \& Sons, Inc., New York-London-Sydney, 1968.

\bibitem{Bryc-Kuznetsov-2021}
Wlodek Bryc and Alexey Kuznetsov.
\newblock {M}arkov limits of steady states of the {KPZ} equation on an
  interval, 2021.
\newblock \url{https://arxiv.org/abs/2109.04462}.

\bibitem{bryc21markov}
W{\l}odek Bryc, Alexey Kuznetsov, Yizao Wang, and Jacek Weso{\l}owski.
\newblock Markov processes related to the stationary measure for the open {KPZ}
  equation.
\newblock {\em Probability Theory Related Fields}, (in press), 2021.
\newblock ArXiv preprint \url{https://arxiv.org/abs/2105.03946}.

\bibitem{bryc10askey}
W{\l}odek Bryc and Jacek Weso{\l}owski.
\newblock Askey--{W}ilson polynomials, quadratic harnesses and martingales.
\newblock {\em Ann. Probab.}, 38(3):1221--1262, 2010.

\bibitem{bryc17asymmetric}
W{\l}odek Bryc and Jacek Weso{\l}owski.
\newblock Asymmetric simple exclusion process with open boundaries and
  quadratic harnesses.
\newblock {\em J. Stat. Phys.}, 167(2):383--415, 2017.

\bibitem{Bryc-Swieca-2018}
W{\l}odzimierz Bryc and Marcin \'Swieca.
\newblock On matrix product ansatz for asymmetric simple exclusion process with
  open boundary in the singular case.
\newblock {\em J. Stat. Phys.}, 177:252--284, 2019.

\bibitem{bryc16local}
W{\l}odzimierz Bryc and Yizao Wang.
\newblock The local structure of {$q$}-{G}aussian processes.
\newblock {\em Probab. Math. Statist.}, 36(2):335--352, 2016.

\bibitem{bryc18dual}
W{\l}odzimierz Bryc and Yizao Wang.
\newblock Dual representations of {L}aplace transforms of {B}rownian excursion
  and generalized meanders.
\newblock {\em Statist. Probab. Lett.}, 140:77--83, 2018.

\bibitem{bryc19limit}
W{\l}odzimierz Bryc and Yizao Wang.
\newblock Limit fluctuations for density of asymmetric simple exclusion
  processes with open boundaries.
\newblock {\em Ann. Inst. Henri Poincar\'{e} Probab. Stat.}, 55(4):2169--2194,
  2019.

\bibitem{calvert2019brownian}
Jacob Calvert, Alan Hammond, and Milind Hegde.
\newblock Brownian structure in the {KPZ} fixed point, 2019.
\newblock \url{https://arxiv.org/abs/1912.00992}.

\bibitem{corwin2022some}
Ivan Corwin.
\newblock Some recent progress on the stationary measure for the open {KPZ}
  equation, 2022.
\newblock \url{https://arxiv.org/abs/2202.01836}.

\bibitem{corwin21stationary}
Ivan Corwin and Alisa Knizel.
\newblock Stationary measure for the open {KPZ} equation, 2021.
\newblock \url{https://arxiv.org/abs/2103.12253}.

\bibitem{corwin2015renormalization}
Ivan Corwin, Jeremy Quastel, and Daniel Remenik.
\newblock Renormalization fixed point of the {KPZ} universality class.
\newblock {\em J. Stat. Phys.}, 160(4):815--834, 2015.

\bibitem{corwin2018open}
Ivan Corwin and Hao Shen.
\newblock Open {ASEP} in the weakly asymmetric regime.
\newblock {\em Communications on Pure and Applied Mathematics},
  71(10):2065--2128, 2018.

\bibitem{dauvergne2021directed}
Duncan Dauvergne, Janosch Ortmann, and B\'alint Vir\'ag.
\newblock The directed landscape, 2021.
\newblock \url{https://arxiv.org/abs/1812.00309}.

\bibitem{derrida06matrix}
Bernard Derrida.
\newblock Matrix ansatz and large deviations of the density in exclusion
  processes.
\newblock In {\em International {C}ongress of {M}athematicians. {V}ol. {III}},
  pages 367--382. Eur. Math. Soc., Z\"urich, 2006.

\bibitem{derrida07nonequilibrium}
Bernard Derrida.
\newblock Non-equilibrium steady states: fluctuations and large deviations of
  the density and of the current.
\newblock {\em J. Stat. Mech. Theory Exp.}, 2007(7):P07023, 45, 2007.

\bibitem{derrida93exact}
Bernard Derrida, Martin~R. Evans, Vincent Hakim, and Vincent Pasquier.
\newblock Exact solution of a {$1$}{D} asymmetric exclusion model using a
  matrix formulation.
\newblock {\em J. Phys. A}, 26(7):1493--1517, 1993.

\bibitem{erdelyi1954fg}
A~Erd{\'e}lyi, W~Magnus, and F~Oberhettinger.
\newblock {\em Tables of integral transforms, vol. {I}}.
\newblock McGraw-Hill, New York, 1954.

\bibitem{gantert2020mixing}
Nina Gantert, Evita Nestoridi, and Dominik Schmid.
\newblock Mixing times for the simple exclusion process with open boundaries,
  2020.
\newblock \url{https://arxiv.org/abs/2003.03781}.

\bibitem{gaspar1990basic}
George Gasper and Mizan Rahman.
\newblock {\em Basic Hypergeometric Series}.
\newblock Cambridge University Press, 1990.

\bibitem{KyprianouOConnell2021}
Andreas~E Kyprianou and Neil O'Connell.
\newblock The {D}oob--{M}c{K}ean identity for stable {L}{\'e}vy processes.
\newblock In {\em A Lifetime of Excursions Through Random Walks and {L}\'evy
  Processes: A Volume in Honour of {R}on {D}oney’s 80th Birthday}, pages
  269--282. Birkh\"auser, 2022.

\bibitem{matetski2016kpz}
Konstantin Matetski, Jeremy Quastel, and Daniel Remenik.
\newblock The {KPZ} fixed point.
\newblock {\em Acta Math.}, 227(1):115--203, 2021.

\bibitem{parekh2019kpz}
Shalin Parekh.
\newblock The {KPZ} limit of {ASEP} with boundary.
\newblock {\em Communications in Mathematical Physics}, 365(2):569--649, 2019.

\bibitem{pimentel2021brownian}
Leandro~PR Pimentel.
\newblock Brownian aspects of the {KPZ} fixed point.
\newblock In {\em In and Out of Equilibrium 3: Celebrating Vladas
  Sidoravicius}, pages 711--739. Springer, 2021.

\bibitem{prudnikov92integrals}
A.~P. Prudnikov, Yu.~A. Brychkov, and O.~I. Marichev.
\newblock {\em Integrals and series. {V}ol. 4}.
\newblock Gordon and Breach Science Publishers, New York, 1992.

\bibitem{quastel2017totally}
Jeremy Quastel and Konstantin Matetski.
\newblock From the totally asymmetric simple exclusion process to the {KPZ}
  fixed point.
\newblock In {\em Random matrices}, volume~26 of {\em IAS/Park City Math.
  Ser.}, pages 251--301. Amer. Math. Soc., Providence, RI, 2019.

\bibitem{quastel2021convergence}
Jeremy Quastel and Sourav Sarkar.
\newblock Convergence of exclusion processes and {KPZ} equation to the {KPZ}
  fixed point, 2021.
\newblock \url{https://arxiv.org/abs/2008.06584}.

\bibitem{Sarkar-Virag-2021}
Sourav Sarkar and B\'alint Vir\'ag.
\newblock Brownian absolute continuity of the {KPZ} fixed point with arbitrary
  initial condition.
\newblock {\em Ann. Probab.}, 49(4):1718--1737, May 2021.

\bibitem{virag2020heat}
B\'alint Vir\'ag.
\newblock The heat and the landscape {I}, 2020.
\newblock \url{ https://arxiv.org/abs/2008.07241 }.

\end{thebibliography}
\end{document}